\documentclass[11pt]{amsart}  

\usepackage{selectp}

\usepackage{paralist}

\usepackage{amsrefs}
\usepackage[all]{xy}\SelectTips{cm}{12}
\usepackage{syntonly}
\usepackage{enumerate}
\usepackage{hyperref}
\usepackage{amsfonts}
\usepackage{amssymb}
\usepackage{amsmath}
\usepackage{amsthm}
\usepackage[inline]{enumitem}
\usepackage{tikz}
\usepackage{graphics}
\usepackage{graphicx}
\usepackage{color}
\usepackage{comment}
\usepackage{varwidth}
\usepackage{lineno}

\usepackage{parskip}

\newtheorem{theorem}{Theorem}
\newtheorem{corollary}[theorem]{Corollary}

\newtheorem{definition}[theorem]{Definition}
\newtheorem{lemma}[theorem]{Lemma}

\newtheorem{nonGlobalClaim}{Claim}[theorem]  

\newtheorem{observation}[theorem]{Observation}

\newtheorem{fact}[theorem]{Fact}

\newtheorem{remark}[theorem]{Remark}

\usepackage[letterpaper,  margin=1.02in]{geometry}

\usepackage{marginnote}

\newcounter{mynote}

\specialcomment{com}{ \color{red} }{\color{black}} 

\specialcomment{instructions}{ \color{blue} }{\color{black}}

\excludecomment{instructions}

\makeatletter
\xydef@\txt@ii#1{\vbox{\vspace*{-7pt}%
 \let\\=\cr
 \tabskip=\z@skip \halign{\relax\hfil\txtline@@{##}\hfil\cr\leavevmode#1\crcr}}}
\makeatother

\begin{document}
\title[Forcing axioms, approachability, and stationary set reflection]{Forcing axioms, approachability, and stationary set reflection}

\author{Sean D. Cox}
\email{scox9@vcu.edu}
\address{
Department of Mathematics and Applied Mathematics \\
Virginia Commonwealth University \\
1015 Floyd Avenue \\
Richmond, Virginia 23284, USA 
}


\thanks{The author gratefully acknowledges support from Simons Foundation grant 318467.}

\begin{abstract}
We prove a variety of theorems about stationary set reflection and concepts related to internal approachability.  We prove that an implication of Fuchino-Usuba relating stationary reflection to a version of Strong Chang's Conjecture cannot be reversed; strengthen and simplify some results of Krueger about forcing axioms and approachability; and prove that some other related results of Krueger are sharp.  We also adapt some ideas of Woodin to simplify and unify many arguments in the literature involving preservation of forcing axioms.
\end{abstract}

\maketitle

\section{Introduction}

Foreman-Todorcevic~\cite{MR2115072} introduced several natural variants of the class IA of internally approachable sets of size $\omega_1$.  These are the classes of internally club sets (IC), the internally stationary sets (IS), and the internally unbounded sets (IU).  The inclusions
\begin{equation}\label{eq_Inclusions}
\text{IA} \subseteq \text{IC} \subseteq \text{IS} \subseteq \text{IU}
\end{equation}
follow from ZFC, and if the Continuum Hypothesis holds, then $\text{IA} =^* \text{IC} =^* \text{IS} =^* \text{IU}$.\footnote{Meaning that for any regular $\theta \ge \omega_2$, for all but nonstationarily many $W \in [H_\theta]^{\omega_1}$, $W$ is in one of those four classes if and only if it is in all of them.}  The chain of inclusions in \eqref{eq_Inclusions} is closely related to Shelah's \emph{Approachability Ideal} $I[\omega_2]$, as follows:
\begin{lemma}[folklore; see remarks after Observation \ref{obs_InternalPartProjects} for a proof]\label{lem_ApproachAndInclusions}
Assume $2^{\omega_1} = \omega_2$.  The assertion that the approachability property fails at $\omega_2$---i.e.\ that $\omega_2 \notin I[\omega_2]$---is equivalent to the assertion that $\text{IU}  \setminus \text{IA}$ is stationary in $\wp_{\omega_2}(H_{\omega_2})$.    In other words, failure of approachability property at $\omega_2$ is equivalent to asserting that \textbf{at least one} of the three inclusions in \eqref{eq_Inclusions} is strict in $P(\wp_{\omega_2}(H_{\omega_2}))/\text{NS}$.
\end{lemma}

In light of Lemma \ref{lem_ApproachAndInclusions}, a separation of adjacent classes in the chain of inclusions from \eqref{eq_Inclusions} can be viewed as a very strong failure of approachability.
 
Foreman and Todorcevic independently proved that the Proper Forcing Axiom (PFA) implies failure of the approachability property at $\omega_2$;\footnote{See the discussion preceding Theorem 5.3 in  K\"onig-Yoshinobu~\cite{MR2050172}.} in particular, PFA implies that at least one of the inclusions in \eqref{eq_Inclusions} must be strict.  Krueger~\cite{MR2332607} improved this, by showing that PFA in fact separates IA from IC in a global fashion.  Given subclasses $\Gamma$ and $\Gamma'$ of $\{ W \ : \ |W|=\omega_1 \subset W \}$ such that $\Gamma \subseteq \Gamma'$, let us say that the inclusion $\Gamma \subseteq \Gamma'$ is \textbf{globally strict} iff $\Big(\Gamma' \cap \wp_{\omega_2}(H_\theta)\Big) \setminus \Big( \Gamma \cap \wp_{\omega_2}(H_\theta)\Big)$ is stationary for every regular $\theta \ge \omega_2$.  Answering a question of Foreman-Todorcevic, Krueger proved (\cite{MR2332607}, \cite{MR2674000}, and \cite{MR2502487}) that \emph{each} of the three inclusions in \eqref{eq_Inclusions} can be globally strict, under various strong forcing axioms.  As mentioned above, PFA globally separates IA from IC, but stronger forcing axioms were used for the following separations:

\begin{theorem}[Krueger~\cite{MR2502487}, Theorem 5.2]\label{thm_Krueger_PFAplus1}
$\text{PFA}^+$ implies that the inclusion $\text{IC} \subseteq \text{IU}$---i.e.\ between the second and fourth class in the chain \eqref{eq_Inclusions}---is globally strict.  In particular, $\text{PFA}^+$ implies there is a disjoint stationary sequence on $\omega_2$.
\end{theorem}  

\begin{theorem}[Krueger; corollary of Theorem 0.2 of \cite{MR2674000} and Theorem 6.3 of \cite{MR2502487}]\label{thm_Krueger_MM}
Martin's Maximum (MM) implies that the inclusion $\text{IS} \subseteq \text{IU}$ is globally strict.  In particular, there is a disjoint club sequence on $\omega_2$.
\end{theorem}

\begin{theorem}[Krueger~\cite{MR2674000}, Theorem 0.3]\label{thm_Krueger_PFAplus2}
$\text{PFA}^{+2}$ implies that the inclusion $\text{IC} \subseteq \text{IS}$ is globally strict.
\end{theorem}

We prove that the assumptions of his Theorems \ref{thm_Krueger_PFAplus1} and \ref{thm_Krueger_MM} are sharp, but the assumption of Theorem \ref{thm_Krueger_PFAplus2} is not:

\begin{theorem}\label{thm_Cox_PFAplus1necessary}
The assumption of $\text{PFA}^+$ in Theorem \ref{thm_Krueger_PFAplus1} cannot be replaced by PFA.  

\end{theorem}

\begin{theorem}\label{thm_Cox_MMnecessary}
The assumption of MM in Theorem \ref{thm_Krueger_MM} cannot be replaced by $\text{PFA}^{+\omega_1}$.
\end{theorem}

\begin{theorem}\label{thm_Cox_DontNeedPlus2}
The conclusion of Theorem \ref{thm_Krueger_PFAplus2} also follows from $\text{PFA}^{+}$ and from Martin's Maximum,\footnote{Separation of IC from IS under Martin's Maximum was claimed in Theorem 4.4 part (3) of Viale~\cite{Viale_GuessingModel}, but the argument given there implicitly used the stronger assumption $\text{MM}^{+2}$.  See Section \ref{sec_MM_Viale} for a discussion.} but not from PFA.
\end{theorem}
Note that the $\text{PFA}^+$ portion of Theorem \ref{thm_Cox_DontNeedPlus2}---i.e.\ that $\text{PFA}^+$ implies that the inclusion $\text{IC} \subseteq \text{IS}$ is globally strict---strengthens both Theorems \ref{thm_Krueger_PFAplus1} and \ref{thm_Krueger_PFAplus2}.  Figure \ref{fig_StrictInclusions} summarizes the theorems above.

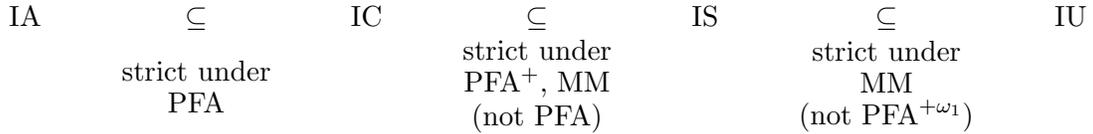
\begin{figure}[!h]
\caption{When the inclusions from \eqref{eq_Inclusions} are globally strict.
}\label{fig_StrictInclusions}
\begin{displaymath}
\xymatrix@R-1.9pc{
\text{IA}   & \subseteq & \text{IC} & \subseteq & \text{IS} & \subseteq & \text{IU} \\
& \txt{strict under\\PFA} &   &  \txt{strict under\\$\text{PFA}^+$, MM \\ (not PFA)} & & \txt{strict under\\MM\\(not $\text{PFA}^{+\omega_1}$)}& 
}
\end{displaymath}
\end{figure}

Foreman and Todorcevic also considered stationary reflection principles for the classes in \eqref{eq_Inclusions}.  Given a (possibly finite) cardinal $\mu \le \omega_1$ and a subclass $\Gamma$ of $\{ W \ : \ |W|=\omega_1 \subset W \}$, let $\text{RP}^\mu_\Gamma$ assert that for all regular $\theta \ge \omega_2$ and every $\mu$-sized collection $\mathcal{S}$ of stationary subsets of $[H_\theta]^\omega$, there is a $W \in \Gamma$ such that $S \cap [W]^\omega$ is stationary in $[W]^\omega$ for every $S \in \mathcal{S}$.  We usually write $\text{RP}_\Gamma$ for $\text{RP}^1_\Gamma$.   Now \eqref{eq_Inclusions} clearly implies 
\begin{equation}\label{eq_RP_implic}
\text{RP}_{\text{IA}} \implies   \text{RP}_{\text{IC}} \implies   \text{RP}_{\text{IS}}  \implies \text{RP}_{\text{IU}}.
\end{equation}

Krueger, answering another question of Foreman-Todorcevic~\cite{MR2115072}, proved that the implication $\text{RP}_{\text{IC}} \ \implies \ \text{RP}_{\text{IS}}$ cannot be reversed; in fact:
\begin{theorem}[Krueger~\cite{MR2674000}]\label{thm_Krueger_RPIS_RPIC}
$\text{RP}^{\omega_1}_{\text{IS}}$ does not imply $\text{RP}_{\text{IC}}$.\footnote{The theorem stated in Theorem 5.1 of \cite{MR2674000} just says that $\text{RP}^{\omega_1}_{\text{IS}}$ does not imply $\text{RP}_{\text{IA}}$, but the proof there clearly shows that even $\text{RP}_{\text{IC}}$ fails in his model.}  
\end{theorem}

Fuchino-Usuba~\cite{FuchinoUsuba} proved another result related to \eqref{eq_RP_implic}.  They introduced a game-theoretic principle denoted $G^{\downarrow \downarrow}$, proved it is equivalent to a version of Strong Chang's Conjecture (see Definition \ref{def_GlobalSCCcofgap} below), and also equivalent to a reflection principle that they did not name, but which we call $\text{RP}_{\text{internal}}$ (see Section \ref{sec_Prelims}).  They proved:
\begin{theorem}[Fuchino-Usuba~\cite{FuchinoUsuba}]\label{thm_FuchinoUsubaImplications}
\[
\text{RP}_{\text{IC}} \implies   \text{RP}_{\text{internal}} \iff  G^{\downarrow \downarrow} \implies \text{RP}_{\text{IS}}.
\]
\end{theorem}

We show below that the left implication of Theorem \ref{thm_FuchinoUsubaImplications} cannot be reversed.  This was already implicit in Krueger's model from Theorem \ref{thm_Krueger_RPIS_RPIC}, but our Theorem \ref{thm_Cox_PreservePFAnegIC} below strengthens and simplifies his result in several ways.  In particular, the model of Theorem \ref{thm_Cox_PreservePFAnegIC}:  
\begin{enumerate*}
 \item satisfies PFA (and a ``plus" version of a fragment of PFA);
 \item can be forced over an arbitrary model of $\text{PFA}^{+\omega_1}$ in a single step; and
 \item satisfies ``diagonal internal reflection to guessing, internally stationary sets" ($\text{DRP}_{\text{internal, GIS}}$), which is a highly simultaneous form of internal stationary reflection to the so-called \emph{guessing} sets used by Viale-Weiss~\cite{VW_ISP} to characterize the principle ISP (see Section \ref{sec_Prelims}).
\end{enumerate*}

In what follows, GIC refers to the class of guessing, internally club sets, and GIS refers to the class of guessing, internally stationary sets.  We also introduce some fragments of standard forcing axioms.  $\boldsymbol{\textbf{PFA}^{+\omega_1}_{H^V_{\omega_2} \notin \textbf{IC}}}$ denotes the forcing axiom $\text{FA}^{+\omega_1}(\Gamma)$, where $\Gamma$ is the class of proper posets that force $H^V_{\omega_2} \notin \text{IC}$;  $\boldsymbol{\textbf{PFA}^{+\omega_1}_{H^V_{\omega_2} \notin \textbf{IA}}}$ is defined similarly (see Section \ref{sec_ForcingAxiomPrelims} for the meaning of $\text{FA}^{+\omega_1}(\Gamma)$).  $\boldsymbol{\textbf{MM}^{+\omega_1}_{H^V_{\omega_2} \notin \textbf{IS}}}$ denotes the forcing axiom $\text{FA}^{+\omega_1}(\Gamma)$, where $\Gamma$ is the class of posets that preserve stationary subsets of $\omega_1$ and force $H^V_{\omega_2} \notin \text{IS}$.

\begin{theorem}\label{thm_PFAplusNotGICimpliesDRPGIS}
$\text{PFA}^{+\omega_1}_{H^V_{\omega_2} \notin \text{IC}}$ implies $\text{DRP}_{\text{internal}, GIS}$ (see Section \ref{sec_Prelims}), a highly simultaneous version of internal stationary reflection to GIS sets (in particular, this implies $\text{RP}^{\omega_1}_{\text{IS}}$).
\end{theorem}

\begin{theorem}\label{thm_Cox_PreservePFAnegIC}
There is a $< \!\! \omega_2$ strategically closed poset that preserves $\text{PFA}^{+\omega_1}_{H^V_{\omega_2} \notin \text{IC}}$ and forces $\neg \text{RP}_{\text{IC}}$.  Moreover, if PFA held in the ground model, then PFA is also preserved.
\end{theorem}

\begin{corollary}\label{cor_FuchUsuba}
$\text{DRP}_{\text{internal, GIS}}$ does not imply $\text{RP}_{\text{IC}}$.  In particular, the implication $\text{RP}_{\text{IC}} \implies \text{RP}_{\text{internal}}$ from Theorem \ref{thm_FuchinoUsubaImplications} cannot be reversed.
\end{corollary}

We also prove the following similar theorems:

\begin{theorem}\label{thm_PFAplusNotGIAimpliesDRPGIC}
$\text{PFA}^{+\omega_1}_{H^V_{\omega_2} \notin \text{IA}}$ implies $\text{DRP}_{GIC}$ (which is equivalent to $\text{DRP}_{\text{internal, GIC}}$; see Observation \ref{obs_RPIC_automaticInternal}).
\end{theorem}

\begin{theorem}\label{thm_Cox_PreservePFAnegIA}
There is a $< \!\! \omega_2$ strategically closed poset that preserves $\text{PFA}^{+\omega_1}_{H^V_{\omega_2} \notin \text{IA}}$ and forces $\neg \text{RP}_{\text{IA}}$. Moreover, if PFA held in the ground model, then PFA is also preserved.
\end{theorem}

\begin{corollary}\label{cor_DRPGIC_notImply_RPIA}
$\text{DRP}_{\text{GIC}}$ does not imply $\text{RP}_{\text{IA}}$.  In particular, the implication $\text{RP}_{\text{IA}} \implies \text{RP}_{\text{IC}}$ from \eqref{eq_RP_implic} cannot be reversed.
\end{corollary}

Figure \ref{fig_RP} summarizes the implications and non-implications.  It shows that for the classes IA, IC, and IS, the maximal form of simultaneous reflection (i.e.\ DRP) for one class does not even imply $\text{RP}^1$ for the next ``nicer" class.  This contrasts greatly with the Foreman-Todorcevic result (Corollary 20 of \cite{MR2115072}), that $\text{RP}^3_{\wp^*_{\omega_2}(V)}$ implies $\text{RP}^1_{\text{Unif}_{\omega_1}}$, where $\wp^*_{\omega_2}(V):= \{ W \ : \ |W|=\omega_1 \subset W \}$ and 
\[
\text{Unif}_{\omega_1} :=  \{ W \in \wp^*(V) \ : \ \text{cf}(\text{sup}(W \cap \kappa)) = \omega_1 \text{ for every regular uncountable } \kappa \in W  \}.
\]

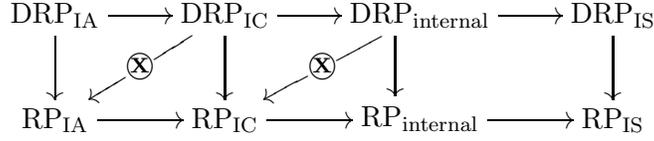
\begin{figure}[!h]
\caption{An arrow indicates an implication, and an arrow with an X indicates a non-implication.  In order to simplify the figure, the non-implications shown do not incorporate the full strength of the theorems above.}
\label{fig_RP}
\begin{displaymath}
\xymatrix{
\text{DRP}_{\text{IA}} \ar@{>}[r] \ar@{>}[d] &  \text{DRP}_{\text{IC}} \ar@{>}[r] \ar@{>}[d]  \ar[dl]|{\textcircled{\scriptsize \textbf{X}}}
 & \text{DRP}_{\text{internal}} \ar[r] \ar@{>}!<-2ex,0ex>;[d]!<-2ex,2ex>    \ar[dl]|{\textcircled{\scriptsize \textbf{X}}}
 & \text{DRP}_{\text{IS}} \ar@{>}[d] \\
\text{RP}_{\text{IA}} \ar@{>}[r] &  \text{RP}_{\text{IC}} \ar@{>}[r] & \txt{$\text{RP}_{\text{internal}}$}  \ar[r]  & \text{RP}_{\text{IS}} 
}
\end{displaymath}
\end{figure}

We now address the implication $\text{RP}_{\text{IS}} \implies \text{RP}_{\text{IU}}$ from \eqref{eq_RP_implic}.  Whether this is reversible is closely related to Question 5.12 of Krueger~\cite{MR2674000}, and to the question of whether the implication $\text{RP}_{\text{internal}} \implies \text{RP}_{\text{IS}}$ from Theorem \ref{thm_FuchinoUsubaImplications} can be reversed.  We do not know the answer to any of those questions, but the following are some partial results that may shed light on this surprisingly difficult problem.  

In light of the way that we separated the various reflection principles above (Theorem \ref{thm_PFAplusNotGICimpliesDRPGIS} through Corollary \ref{cor_DRPGIC_notImply_RPIA}), it is natural to conjecture that $\text{MM}^{+\omega_1}_{H^V_{\omega_2} \notin \text{IS}}$ should imply $\text{DRP}_{\text{IU}}$, and that $\text{MM}^{+\omega_1}_{H^V_{\omega_2} \notin \text{IS}}$ should be preserved by a forcing that kills $\text{RP}_{\text{IS}}$.  This would have separated $\text{RP}_{\text{IS}}$ from $\text{RP}_{\text{IU}}$ in a manner similar to the earlier separation results.  However, $\text{MM}^{+\omega_1}_{H^V_{\omega_2} \notin \text{IS}}$ does \emph{not} imply $\text{DRP}_{\text{IU}}$, and in fact does not even imply the weakest generalized reflection principle of all, as the next theorem shows.  The notation $\text{WRP}$ is conventionally used to denote  $\text{RP}_\Gamma$, where $\Gamma$ is the entire class $\{W \ : \ |W|=\omega_1 \subset W \}$.

\begin{theorem}\label{thm_MM_notIS_notImplyRPIU}
$\text{MM}^{+\omega_1}_{H^V_{\omega_2} \notin \text{IS}}$ is consistent with failure of $\text{WRP}(\omega_2)$.
\end{theorem}

Recall that $\text{MM}^{+\omega_1}_{H^V_{\omega_2} \notin \text{IS}}$ is (the $+\omega_1$ version of) the forcing axiom for the class of posets that preserve stationary subsets of $\omega_1$ and force $H^V_{\omega_2} \notin \text{IS}$.  A familiar poset in this class is Namba forcing.  Martin's Maximum implies that Namba forcing is semiproper, which in turn implies $\text{WRP}(\omega_2)$.\footnote{By Shelah~\cite{MR1623206} and Todorcevic~\cite{MR1261218}.}  Also, by Sakai~\cite{MR2387938}, $\text{WRP}(\omega_2)$ is equivalent to the \emph{semi}-stationary reflection principle for $\omega_2$, denoted $\text{SSR}(\omega_2)$.  The following corollary of Theorem \ref{thm_MM_notIS_notImplyRPIU} may be of independent interest:
\begin{corollary}
The axiom $\text{FA}^{+\omega_1}(\text{Namba forcing})$ does not imply $\text{SSR}(\omega_2)$.
\end{corollary}

Although $\text{MM}^{+\omega_1}_{H^V_{\omega_2} \notin \text{IS}}$ does not imply even $\text{WRP}(\omega_2)$, it does imply a diagonal kind of ordinal reflection.  The principle $\text{DRP}^{\text{cof}(\omega)}_{\text{GIU}}$ is a weakening of $\text{DRP}_{\text{GIU}}$ where one only asks for diagonal reflection of stationary subsets of $\theta \cap \text{cof}(\omega)$, rather than stationary subsets of $[\theta]^\omega$.  We prove:
\begin{theorem}\label{thm_MMplusnotISimpliesORD_DRP}
$\text{MM}^{+\omega_1}_{H^V_{\omega_2} \notin \text{IS}}$ implies $\text{DRP}^{\text{cof}(\omega)}_{\text{GIU}}$.
\end{theorem}

\begin{corollary}
$\text{DRP}^{\text{cof}(\omega)}_{\text{GIU}}$ does not imply $\text{WRP}(\omega_2)$.
\end{corollary}

To aid in the proofs of most of the theorems above, we adapt some ideas of Woodin to prove a very general theorem (Theorem \ref{thm_PreservationFA}) about preservation of forcing axioms.  Theorem \ref{thm_PreservationFA} makes such preservation arguments more closely resemble arguments about lifting large cardinal embeddings.  Theorem \ref{thm_PreservationFA} consolidates a variety of results in the literature about forcing axiom preservation into a single framework (see examples in Section \ref{sec_PreservViaEmbed}) and may be of use in other applications.

\begin{theorem}[Forcing Axiom Preservation Theorem]\label{thm_PreservationFA}
Assume $\Gamma$ is a class of posets that is closed under restrictions (see Definition \ref{def_Nice}), $\mathbb{P}$ is a partial order, and $\dot{\Delta}$ is a $\mathbb{P}$-name for a (definable) class of posets that is also closed under restrictions.  Assume $\mu \le \omega_1$ is a cardinal (possibly $\mu = 0$), and that $\text{FA}^{+\mu}(\Gamma)$ holds.  Assume also that for every $\mathbb{P}$-name for a poset $\dot{\mathbb{Q}} \in \dot{\Delta}$ and every $\mathbb{P}*\dot{\mathbb{Q}}$-name $\langle \dot{S}_i \ : \ i < \mu \rangle$ for a sequence of $\mu$ many stationary subsets of $\omega_1$, there exists a $\mathbb{P}*\dot{\mathbb{Q}}$-name $\dot{\mathbb{R}}$ (possibly depending on $\dot{\mathbb{Q}}$ and $\dot{\vec{S}}$) such that:

\begin{enumerate}
 \item\label{item_3stepinGamma} $\mathbb{P}*\dot{\mathbb{Q}}*\dot{\mathbb{R}} \in \Gamma$
 \item\label{item_R_PreserveStat} $\boldsymbol{\mathbb{P}*\dot{\mathbb{Q}}}$ forces that $\boldsymbol{\dot{\mathbb{R}}}$ preserves the stationarity of $\dot{S}_i$ for every $i < \mu$; 
 \item\label{item_ClauseWhenever} \textbf{If} $j: V \to N$ is a generic elementary embedding, $\theta \ge |\mathbb{P}*\dot{\mathbb{Q}}*\dot{\mathbb{R}}|^+$ is regular in $V$, and 
 \begin{enumerate}
  \item $H^V_\theta$ is in the wellfounded part of $N$ (which we assume has been transitivized);
  \item $j[H^V_\theta] \in N$ and has size $\omega_1$ in $N$;
  \item $\text{crit}(j) = \omega_2^V$; and
  \item There exists a $G*H*K$ in $N$ such that:
  \begin{enumerate}
   \item  $G*H*K$ is $(H_\theta^V, \mathbb{P}*\dot{\mathbb{Q}}*\dot{\mathbb{R}})$-generic;
   \item  $N \models$ ``$(\dot{S}_i)_{G*H}$ is stationary for all $i < \mu$";
    \end{enumerate}
 \end{enumerate}
\textbf{then} $N$ believes that $j[G]$ has a lower bound in $j(\mathbb{P})$.
\end{enumerate}

Then $V^{\mathbb{P}} \models \text{FA}^{+\mu}(\dot{\Delta})$.
\end{theorem}

The theorem above is stated in a general form to accommodate both
\begin{itemize}
 \item the ``plus" versions of forcing axioms; and
 \item situations where one only wants to preserve a fragment of a forcing axiom---e.g.\ if $\Gamma$ is the class of proper forcings, but $\dot{\Delta}$ names the class of \emph{totally} proper posets in $V^{\mathbb{P}}$.  
\end{itemize}
In many situations, however, one wants for $\mathbb{P}$ to just preserve, say, PFA.  In this situation, the $\mu$ from Theorem \ref{thm_PreservationFA} is zero, and Theorem \ref{thm_PreservationFA} tells us that it suffices to show that for every $\mathbb{P}$-name $\dot{\mathbb{Q}}$ for a proper poset, there exists a $\mathbb{P}*\dot{\mathbb{Q}}$-name $\dot{\mathbb{R}}$ (possibly depending on $\dot{\mathbb{Q}}$) such that:
\begin{enumerate}
 \item $\mathbb{P}*\dot{\mathbb{Q}}*\dot{\mathbb{R}}$ is proper;
  \item \textbf{If} $j: V \to N$ is a generic elementary embedding, $\theta \ge |\mathbb{P}*\dot{\mathbb{Q}}*\dot{\mathbb{R}}|^+$ is regular in $V$, and 
 \begin{enumerate}
  \item $H^V_\theta$ is in the wellfounded part of $N$;
  \item $j[H^V_\theta] \in N$ and has size $\omega_1$ in $N$;
  \item $\text{crit}(j) = \omega_2^V$; and
  \item There exists a $G*H*K$ in $N$ that is $(H_\theta^V, \mathbb{P}*\dot{\mathbb{Q}}*\dot{\mathbb{R}})$-generic;

 \end{enumerate}
\textbf{then} $N$ believes that $j[G]$ has a lower bound in $j(\mathbb{P})$.
\end{enumerate}

In addition to using Theorem \ref{thm_PreservationFA} to prove our own results, Section \ref{sec_ExamplesFApres} provides several examples from the literature that can be viewed as instances of Theorem \ref{thm_PreservationFA}.

Section \ref{sec_Prelims} provides the relevant background.  Section \ref{sec_ControlInternal} provides some key theorems for precisely controlling the ``internal part" of an elementary submodel of size $\omega_1$, while also ensuring that the elementary submodel will be a guessing set.  These arguments frequently make use of Todorcevic's ``$\in$-collapse" poset of finite conditions (e.g.\ as in \cite{MR763902}), modified by a device introduced by Neeman~\cite{MR3201836} to ensure continuity of the generic chain.  Section \ref{sec_PreservViaEmbed} proves the forcing axiom preservation theorem mentioned above.  Section \ref{sec_SeparationApproach} proves Theorems \ref{thm_Cox_PFAplus1necessary}, \ref{thm_Cox_MMnecessary}, and \ref{thm_Cox_DontNeedPlus2}.  Section \ref{sec_SeparateStatReflect} proves the remaining theorems from the introduction (about RP and the forcing axioms $\text{PFA}^{+\omega_1}_{H^V_{\omega_2} \notin \text{IC}}$ etc.).  Section \ref{sec_ClosingRemarks} includes some questions and closing remarks.

\section{Preliminaries}\label{sec_Prelims}

Unless otherwise noted, all notation and terminology comes from Jech~\cite{MR1940513}.  For $m < n < \omega$, $S^n_m$ denotes the set of $\alpha < \omega_n$ such that $\text{cf}(\alpha) = \omega_m$.

\subsection{Classes of $\omega_1$ sized sets}\label{sec_ClassesOfSets}
We will use $\wp^*_{\omega_2}(V)$ to denote the class $\{ W \ : \ |W|=\omega_1 \subset W \}$ (the star superscript indicates that we are \emph{not} including Chang-like structures, i.e.\ $\wp^*_{\omega_2}(V)$ does not include $W$ of size $\omega_1$ such that $|W \cap \omega_1|=\omega$).  Foreman-Todorcevic~\cite{MR2115072} defined several weakenings of internal approachability; e.g.\ they called a $W \in \wp^*_{\omega_2}(V)$ \textbf{internally stationary} iff $W \cap [W]^\omega$ is stationary in $[W]^\omega$.  For the proofs in this article it will be convenient to use the following equivalent definitions.  If $W \in \wp^*_{\omega_2}(V)$, a \textbf{filtration of $\boldsymbol{W}$} is any $\subseteq$-increasing and continuous sequence $\vec{N}=\langle N_i \ : \ i < \omega_1 \rangle$ of countable sets such that $W = \bigcup_{i < \omega_1} N_i$.  It is easy to see that if $\vec{N}$ and $\vec{M}$ are both filtrations of $W$, then $N_i = M_i$ for all but nonstationarily many $i < \omega_1$.  The \textbf{internal part of $\boldsymbol{W}$}, denoted $\boldsymbol{ \textbf{int}(W)}$, is the equivalence class
\[
[\{ i < \omega_1 \ : \ N_i \in W \}]
\]
in the boolean algebra $\wp(\omega_1)/\text{NS}_{\omega_1}$, where $\vec{N}$ is any filtration of $W$.  Since any two filtrations of $W$ agree on a club subset of $\omega_1$, the choice of the filtration does not matter.  We will often abuse terminology and refer to a subset $T \subseteq \omega_1$ as the internal part of $W$, when really we mean that the equivalence class of $T$ modulo $\text{NS}_{\omega_1}$ is the internal part of $W$.  Similarly, when we say ``the internal part of $W$ contains $T$" we mean $[T] \le \text{int}(W)$ in the boolean algebra $\wp(\omega_1)/\text{NS}_{\omega_1}$.  The \textbf{external part of $\boldsymbol{W}$} is defined to be $\omega_1 \setminus \text{int}(W)$.  If $W \in \wp^*_{\omega_2}(V)$, $W$ is called \textbf{internally approachable} if there exists a filtration $\vec{N}=\langle N_i \ : \ i < \omega_1 \rangle$ of $W$ such that $\vec{N} \restriction i \in W$ for every $i < \omega_1$;
 \textbf{internally club} if $\text{int}(W)$ contains a club subset of $\omega_1$; and \textbf{internally stationary} if $\text{int}(W)$ is a stationary subset of $\omega_1$.  We use $\text{IA}$, $\text{IC}$, and $\text{IS}$, to denote the class of internally approachable, internally club, and internally stationary sets, respectively.   We will also sometimes refer to the class $\text{IU}$ of \textbf{internally unbounded} sets, which is the class of $W \in \wp^*_{\omega_2}(V)$ such that $W \cap [W]^\omega$ is $\subseteq$-cofinal in $[W]^\omega$.\footnote{$W \in \text{IU}$ can also be characterized by saying that there exists a filtration $\vec{N}$ of $W$ such that $N_i \in W$ for unboundedly many $i < \omega_1$.  If $W \in \text{IU} \setminus \text{IS}$, the set of such $i$ is nonstationary in $\omega_1$ for any filtration (and even empty for some filtrations of $W$), and hence $\text{int}(W) = [\emptyset]$ in the boolean algebra $\wp(\omega_1)/\text{NS}_{\omega_1}$.  However, the assertion $\text{int}(W) = [\emptyset]$ does not characterize ``$W \in \text{IU} \setminus \text{IS}$" even among those $W$ of uniform cofinality $\omega_1$, because there are always stationarily many $W \in \wp^*_{\omega_2}(H_{\aleph_{\omega+1}})$ that have uniform cofinality $\omega_1$, yet are not internally unbounded (this is due to Zapletal; see Foreman-Magidor~\cite{MR1846032}). }

\begin{lemma}\label{lem_InternalPartAbsolute}
Suppose $W \in \wp^*_{\omega_2}(V)$ has internal part $T$ and external part $T^c$ (we make no assumptions about stationarity or costationarity of $T$ here).  Then this remains true in any outer model where $\omega_1$ is not collapsed.
\end{lemma}
\begin{proof}
Fix any filtration $\vec{N}=\langle N_i \ : \ i < \omega_1 \rangle$ of $W$.  Then by definition of internal and external parts, there is a club $C \subseteq \omega_1$ such that 
\[
T \cap C \subseteq \{ i < \omega_1 \ : \ N_i \in W \}
\]
and
\[
T^c \cap C \subseteq \{ i < \omega_1 \ : \ N_i \notin W \}.
\]
If $V'$ is an outer model of $V$ with the same $\omega_1$, then $C$ is of course still club in $V'$, so $[T]=[T \cap C]$ and $[T^c] =[T^c \cap C]$ in the $\wp(\omega_1)/\text{NS}_{\omega_1}$ of $V'$.  So the above two containments witness that $V'$ believes $T$ is the internal part, and $T^c$ is the external part, of $W$.
\end{proof}

We also observe:
\begin{observation}\label{obs_InternalPartProjects}
If $M$ is a transitive $\text{ZF}^-$ model of size $\omega_1$ with internal part $T$, and $\mu \in [\omega_2^M, M \cap \text{ORD}]$ is regular in $M$, then the internal part of $(H_\mu)^M$ contains $T$.   
\end{observation}
\begin{proof}
Let $\vec{Q} = \langle Q_i \ : \ i < \omega_1 \rangle$ be a filtration of $M$, and let $C$ be club in $\omega_1$ such that $Q_i \in M$ for every $i \in T \cap C$.  Then $\langle Q_i \cap (H_\mu)^M \ : \ i < \omega_1 \rangle$ is a filtration of $(H_\mu)^M$, and if $i \in T \cap C$ then $Q_i \in M$ and hence $Q_i \cap (H_\mu)^M \in M$. 
\end{proof}

Although we will not use it in this paper, we provide a brief sketch of the proof of the folklore Lemma \ref{lem_ApproachAndInclusions} from the introduction.  Similar arguments are implicit in work of Shelah, and somewhat more explicit in Foreman-Magidor~\cite{MR1450520}.  By the assumption $2^{\omega_1} = \omega_2$, we can fix a bijection $\Phi: \omega_2 \to H_{\omega_2}$.  It is routine to check that
\begin{equation}\label{eq_CharIUPhi}
\text{IU} \cap \wp_{\omega_2}(H_{\omega_2}) =^* \{ \Phi[\gamma] \ : \ \gamma \in S^2_1 \}
\end{equation}
where the $=^*$ means ``equal mod NS in $\wp_{\omega_2}(H_{\omega_2})$".  Assume first that $\omega_2 \in I[\omega_2]$; this implies (see Foreman~\cite{MattHandbook}) that for a sufficiently large regular $\theta$, there is a first order structure $\mathfrak{A}=(H_\theta,\in,\dots)$ in a countable language and an $\omega_1$-club $D \subseteq S^2_1$ such that for every $\gamma \in D$, $\text{Sk}^{\mathfrak{A}}(\gamma) \cap \omega_2 = \gamma$ and there exists a strictly increasing sequence $\vec{\beta}^\gamma = \langle \beta^\gamma_i \ : \ i < \omega_1 \rangle$ that is cofinal in $\gamma$, and every proper initial segment of $\vec{\beta}^\gamma$ is an element of $W_\gamma:=\text{Sk}^{\mathfrak{A}}(\gamma)$.  We can without loss of generality assume that $\mathfrak{A}$ includes a predicate for $\Phi$, and hence for $W_\gamma \cap H_{\omega_2} = \Phi[\gamma]$ for every $\gamma \in D$.  By \eqref{eq_CharIUPhi}, together with the assumption that $D$ is almost all of $S^2_1$, to show that $\text{IA} =^* \text{IU}$ it will suffice to show that $W_\gamma \cap H_{\omega_2} \in \text{IA}$ for every $\gamma \in D$.  And for any $\gamma \in D$, $\langle i \cup \Phi[ \{ \beta^\gamma_j \ : \ j < i \}  ] \ : \ i < \omega_1 \rangle$ can easily be shown to be a filtration of $\Phi[\gamma]=W_\gamma \cap H_{\omega_2}$ witnessing its internal approachability.  The other direction of Lemma \ref{lem_ApproachAndInclusions} is easier; we leave this to the reader.

\subsection{Stationary reflection principles}\label{sec_prelims_statreflect}

Given a regular cardinal $\theta \ge \omega_2$, a subclass $\Gamma$ of $\wp^*_{\omega_2}(V)$, and a (possibly finite) cardinal $\mu \le \omega_1$,  $\boldsymbol{\textbf{RP}^\mu_\Gamma(\theta)}$ is the assertion that for every $\mu$-sized collection $\mathcal{S}$ of stationary subsets of $[H_\theta]^\omega$, there is a $W \in \Gamma$ such that $S \cap [W]^\omega$ is stationary for every $S \in \mathcal{S}$.  The principle $\boldsymbol{\textbf{RP}^\mu_{\textbf{internal, } \Gamma}(\theta)}$---which was isolated, but not named, in Fuchino-Usuba~\cite{FuchinoUsuba}---asserts that for every $\mu$-sized collection $\mathcal{S}$ of stationary subsets of $[H_\theta]^\omega$, there is a $W \in \Gamma$ such that $S \cap W \cap [W]^\omega$ (not just $S \cap [W]^\omega$) is stationary for every $S \in \mathcal{S}$.  

If $\Gamma$ is not specified in the subscript of RP, it is understood to be $\wp^*_{\omega_2}(V)$.  We also typically write $\text{RP}$ instead of $\text{RP}^1$.   $\text{RP}^\mu_\Gamma$ means that $\text{RP}^\mu_\Gamma(\theta)$ holds for every regular $\theta \ge \omega_2$.  Similar notational shortcuts apply to all the reflection principles we mention.  Notice that:

\begin{observation}\label{obs_RPIC_automaticInternal}
$\text{RP}_{\text{IC}}$ is equivalent to $\text{RP}_{\text{internal, IC}}$, since if $W \in \text{IC}$ and $S \cap [W]^\omega$ is stationary, then $S \cap W \cap [W]^\omega$ is stationary as well.
\end{observation}

We now recall the diagonal versions of stationary reflection, as introduced in \cite{DRP}.  $\boldsymbol{\textbf{DRP}_\Gamma(\theta)}$  asserts that there are stationarily many $W \in \wp_{\omega_2}(H_{(2^\theta)^+})$ such that $W \cap H_\theta \in \Gamma$ and $S \cap [W \cap H_\theta]^\omega$ is stationary in $[W \cap H_\theta]^\omega$ whenever $S \in W$ and $S$ is stationary in $[H_\theta]^\omega$.   We also define a diagonal version of Fuchino-Usuba's internal reflection:     $\boldsymbol{\textbf{DRP}_{\textbf{internal, }\Gamma}(\theta)}$ is defined the same way as $\boldsymbol{\textbf{DRP}_\Gamma(\theta)}$, except we require that $S \cap W \cap [W \cap H_\theta]^\omega$, not just $S \cap [W \cap H_\theta]^\omega$, is stationary (for every $S \in W$ that is stationary in $[H_\theta]^\omega$).

\begin{remark}\label{rem_Larson}
The principle DRP easily implies $\text{RP}^{\omega_1}$, and in Cox~\cite{DRP} it was asked if they are equivalent.  The author subsequently noticed that Larson~\cite{MR1782117} gives a model where $\text{RP}^{\omega_1}$ holds, but DRP fails.  So DRP is strictly stronger than $\text{RP}^{\omega_1}$.
\end{remark}
   
We also define a weaker kind of diagonal reflection.  If $\Gamma$ is a subclass of $\wp^*_{\omega_2}(V)$, $\text{DRP}^{\text{cof}(\omega)}_\Gamma$ asserts that for every regular $\theta \ge \omega_2$, there are stationarily many $W \in \wp^*_{\omega_2}\big(  H_{(2^\theta)^+} \big)$ such that for every $R \in W$ that is a stationary subset of $\theta \cap \text{cof}(\omega)$, $R \cap \text{sup}(W \cap \theta)$ is stationary in $\text{sup}(W \cap \theta)$.  This is a consequence of ``weak DRP" ($\text{wDRP}_\Gamma$) introduced in \cite{DRP}; whether these two principles are equivalent is not known.

We now mention a few facts that, although we will not use them in this paper, illustrate that the notion of internal stationary reflection (and its diagonal version) are quite natural and related to several other well-studied topics.  The following version of Strong Chang's Conjecture was isolated (but not named) in Fuchino-Usuba~\cite{FuchinoUsuba}.  It is a stronger version of the principle $\text{SCC}^{\text{cof}}_{\text{gap}}$ considered in Cox~\cite{Cox_Nonreasonable}.

\begin{definition}\label{def_GlobalSCCcofgap}
\textbf{Global $\boldsymbol{\text{SCC}^{\text{cof}}_{\text{gap}}}$} is the following assertion:  for all sufficiently large regular $\theta$ and all wellorders $\Delta$ on $H_\theta$ and every countable $M \prec (H_\theta,\in, \Delta)$, there are $\subseteq$-cofinally many $W \in [H_\theta]^{\omega_1}$ such that:
\begin{enumerate}
 \item $\omega_1 \subset W$;
 \item Letting $M(W)$ denote $ \text{Sk}^{(H_\theta,\in,\Delta)}(M \cup \{ W \})$, we have
 \[
 M(W) \cap W = M
 \]
\end{enumerate}
\end{definition}

We note that the version of Strong Chang's Conjecture isolated by Doebler-Schindler~\cite{MR2576698} has a similar characterization, where the $M(W) \cap W = M$ requirement from Definition \ref{def_GlobalSCCcofgap} is weakened to the requirement that $M(W) \cap W \sqsupseteq M$ (equivalently, that $M(W) \cap \omega_1 = M \cap \omega_1$).  Doebler-Schindler's version is equivalent to the $\dagger$ principle, which in turn is equivalent to Semistationary set reflection (see \cite{MR2576698}).

The relevance of Global $\text{SCC}^{\text{cof}}_{\text{gap}}$ to the current paper is the following theorem:
\begin{theorem}[Fuchino-Usuba~\cite{MR2576698}]
$\text{RP}_{\text{internal}}$ is equivalent to Global $\text{SCC}^{\text{cof}}_{\text{gap}}$.
\end{theorem}

The following lemma characterizes internal, diagonal reflection.  It is a situation one often encounters when generically lifting a large cardinal embedding $j: V \to M$ to domain $V^{\mathbb{P}}$, when $j(\mathbb{P})$ is proper in $M$ (but not necessarily in $V$).

\begin{lemma}\label{lem_Char_InternalDRPGamma}
Let $\Gamma$ be a subclass of $\wp^*_{\omega_2}(V)$.  The following are equivalent:
\begin{enumerate}
 \item\label{item_DRPinternal} $\text{DRP}_{\text{internal}, \ \Gamma}$
 \item\label{item_Condense} For every regular $\theta \ge \omega_2$ there are stationarily many $W \in \wp^*_{\omega_2}(H_{(2^\theta)^+})$ such that:
 \begin{enumerate}
  \item  $W \cap H_\theta \in \Gamma$; and
  \item $H_W$ is correct about stationary subsets of $[\sigma_W^{-1}(\theta)]^\omega$, where $\sigma_W: H_W \to W \prec H_{(2^\theta)^+}$ is the inverse of the Mostowski collapse of $W$.
  \end{enumerate}
 \item\label{item_DRPintEmbedding} For every regular $\theta \ge \omega_2$ there is a generic elementary embedding $j: V \to N$ such that:
 \begin{enumerate}
  \item $H^V_{(2^\theta)^+}$ is in the wellfounded part of $N$;
  \item $j \restriction H^V_{(2^\theta)^+} \in N$;
  \item $N \models $ $j[H^V_\theta]$ is a member of $j(\Gamma)$;

  \item $\text{crit}(j) = \omega_2^V$
  \item $N \models$ ``$H^V_{(2^\theta)^+}$ is correct about stationary subsets of $[\theta]^\omega$".
 \end{enumerate}
\end{enumerate}
\end{lemma}

The equivalence of \ref{item_DRPinternal} with \ref{item_Condense} is almost identical to the proof of Theorem 3.6 of Cox~\cite{DRP}, and the equivalence of \ref{item_Condense} with \ref{item_DRPintEmbedding} is a generic ultrapower argument, closely resembling the proof of Theorem 8 of Cox~\cite{Cox_FA_Ideals}.  We refer the reader to those sources.\footnote{ Those two cited proofs deal with the particular class $\Gamma = \text{IC}$, but go through for internal reflection.  The key use of the class IC in those proofs was the above Observation \ref{obs_RPIC_automaticInternal}.  }

\subsection{Forcing Axioms and partially generic filters}\label{sec_ForcingAxiomPrelims}

Given a class $\Gamma$ of partial orders, and a (possibly finite) cardinal $\mu \le \omega_1$,  $\boldsymbol{\textbf{FA}^{+\mu}(\Gamma)}$ is the assertion:  whenever $\mathbb{P} \in \Gamma$, $\mathcal{D}$ is an $\omega_1$-sized collection of dense subsets of $\mathbb{P}$, and $\langle \dot{S}_i \ : \ i < \mu \rangle$ is a $\mu$-length list of $\mathbb{P}$-names for stationary subsets of $\omega_1$, then there is a filter $g \subset \mathbb{P}$ such that $g \cap D \ne \emptyset$ for every $D \in \mathcal{D}$, and for every $i < \mu$ the following set is stationary in $\omega_1$:
\[
(\dot{S}_i)_g:= \{ \xi < \omega_1 \ : \ \exists p \in g \ \ p \Vdash \check{\xi} \in \dot{S}_i   \}.
\]
In the literature $\text{FA}^{+\omega_1}(\Gamma)$ is often denoted $\text{FA}^{++}(\Gamma)$, but we do not use that convention (since we will need to deal both with the case $\mu = 2$ and the case $\mu = \omega_1$).  $\textbf{FA}\boldsymbol{(\Gamma)}$ is as defined above, but without mentioning the names for stationary sets (so $\text{FA}(\Gamma)$ is the same as $\text{FA}^{+0}(\Gamma)$).

If $\mathbb{Q}$ is a poset and $\mathbb{Q} \in W \prec H_\theta$, we say that \textbf{$\boldsymbol{g}$ is a $\boldsymbol{(W,\mathbb{Q})}$-generic filter} iff $g$ is a filter on $W \cap \mathbb{Q}$ and $g \cap D \cap W \ne \emptyset$ for every $D \in W$ that is dense in $\mathbb{Q}$.  If $\mu \le \omega_1$ and $\dot{\vec{S}} = \langle \dot{S}_i \ : \ i < \mu \rangle$ is a sequence of $\mathbb{Q}$-names for stationary subsets of $\omega_1$ such that $\dot{\vec{S}} \in W$, we say that \textbf{$\boldsymbol{g}$ is an $\boldsymbol{\dot{\vec{S}}}$-correct, $\boldsymbol{(W,\mathbb{Q})}$-generic filter} if $g$ is a $(W,\mathbb{Q})$-generic filter as defined above, and for every $i < \mu$,
\[
(\dot{S}_i)_g:= \{ \alpha < \omega_1 \ : \ \exists q \in g \ q \Vdash \ \check{\alpha} \in \dot{S}_i \}
\]
is (in $V$) a stationary subset of $\omega_1$.

Let $\sigma_W:H_W \to W \prec H_\theta$ be the inverse of the Mostowski collapse of $W$, and suppose $g \subset W \cap \wp(\mathbb{Q})$.   It is easy to see that $g$ is $(W,\mathbb{Q})$-generic if and only if $\sigma_W^{-1}[g]$ is an $(H_W, \sigma^{-1}(\mathbb{Q}))$-generic filter in the usual sense.  Furthermore, if $\mu \subset W$, $\dot{\vec{S}} \in W$, and $g$ is a $(W,\mathbb{Q})$-generic filter, then $\dot{S}_i \in W$ for every $i < \mu$, and $g$ is $\dot{\vec{S}}$-correct iff for every $i < \mu$, the evaluation of $\sigma_W^{-1}(\dot{S}_i)$ by $\sigma_W^{-1}[g]$ is (in $V$) a stationary subset of $\omega_1$.

\subsection{Guessing models and strongly proper club shooting}\label{sec_Guessing}

A pair of transitive $\text{ZF}^-$ models $(M,N)$ has the \textbf{$\boldsymbol{\omega_1}$-approximation property} iff for every $X \in M$ and every $A \in \wp(X) \cap N$, if $A \cap z \in M$ for every $z \in M$ such that $M \models$ ``$z$ is countable", then $A \in M$.  We say that $(M,N)$ has the \textbf{$\boldsymbol{\omega_1}$-covering property }  iff for every $z \in N$ that is countable in $N$, there is a $z' \in M$ that is countable in $M$ such that $z \subseteq z'$.  A poset $\mathbb{P}$ has the \textbf{$\boldsymbol{\omega_1}$ approximation property } iff  it forces that $(V,V^{\mathbb{P}})$ has the $\omega_1$ approximation property; and has the \textbf{$\boldsymbol{\omega_1}$-covering property} iff it forces that $(V,V^{\mathbb{P}})$ has the $\omega_1$-covering property.  The following fact is often used:
\begin{fact}[Viale-Weiss~\cite{VW_ISP}]\label{fact_CoverApproxClosed2step}
If $\mathbb{P}$ has the $\omega_1$ approximation and covering properties, and $\Vdash_{\mathbb{P}}$ ``$\dot{\mathbb{Q}}$ has the $\omega_1$ approximation and covering properties", then $\mathbb{P}*\dot{\mathbb{Q}}$ has the $\omega_1$ approximation and covering properties.
\end{fact}

A set $W$ is called an \textbf{$\boldsymbol{\omega_1}$-guessing model} if $W \in \wp^*_{\omega_2}(V)$, $\big(W, \in \cap (W \times W) \big) \models \text{ZF}^-$, and $(H_W,V)$ has the $\omega_1$-approximation property, where $H_W$ is the transitive collapse of $W$.\footnote{This is equivalent to the original definition of guessing model from Viale-Weiss~\cite{VW_ISP}.}   We will often just say ``guessing model" instead of ``$\omega_1$-guessing model", and use either $\boldsymbol{G}$ or (when there is risk of confusion) $\boldsymbol{G_{\omega_1}}$ to denote the class of guessing models.  Viale-Weiss~\cite{VW_ISP} proved that Weiss' generalized tree property principle $\text{ISP}(\omega_2)$ is equivalent to the assertion that for every regular $\theta \ge \omega_2$, the set of guessing models is stationary in $\wp^*_{\omega_2}(H_\theta)$.  We say that $W$ is an \textbf{indestructible guessing model (in $\boldsymbol{V}$)} if, letting $\theta_W$ be the least regular cardinal such that $W \in H^V_{\theta_W}$, then $W$ remains a guessing model in any $\text{ZF}^-$ model $(N, \in_N)$ such that $H^V_{\theta_W}$ is an element of the  wellfounded part of $N$ and $\omega_1^V = \omega_1^N$ (here $N$ may be external to $V$, and we even allow that $V$ can be a set from $N$'s point of view.).  This of course isn't a first order statement over $V$, but in typical contexts there are first-order substitutes for this notion, e.g.\ the presence of specializing functions in $H^V_\theta$ that guarantee such indestructibility (see Proposition 4.4 of \cite{Cox_Krueger_2}).    

We let GIU denote the (possibly empty) class of guessing models that are also internally unbounded.
\begin{theorem}[Viale-Weiss~\cite{VW_ISP}; see also Proposition 4.4 of \cite{Cox_Krueger_2}]\label{thm_Sealing}
Assume $\theta \ge \omega_2$ is regular, and $\mathbb{P}$ is a poset that forces $H^V_\theta \in \text{GIU}$.\footnote{A sufficient condition for this is if $\mathbb{P}$ collapses $H^V_\theta$ to size $\omega_1$, and has the $\omega_1$ covering and approximation properties.}  Then there is a $\mathbb{P}$-name $\dot{\mathbb{S}}(H^V_\theta)$ for a c.c.c.\ poset such that $\mathbb{P}*\dot{\mathbb{S}}(H^V_\theta)$ forces $H^V_\theta$ to be \textbf{indestructibly} guessing.
\end{theorem} 

\begin{corollary}
Suppose $W \prec H_{(2^\theta)^+}$, $W \in \wp^*_{\omega_2}(V)$, and $\mathbb{P} \in W$ is a poset forcing $H^V_\theta \in \text{GIU}$.   Let $\dot{\mathbb{S}}(H^V_\theta)$ be as in the conclusion of Theorem \ref{thm_Sealing}.  If there exists a $\big(W,\mathbb{P}*\dot{\mathbb{S}}(H^V_\theta)\big)$-generic filter, then $W \cap H_\theta$ is indestructibly guessing.
\end{corollary}
\begin{proof}
Suppose $g*h$ is such a $W$-generic filter.  Let $\pi: W \to H_W$ be the transitive collapse of $W$, $\theta_W:=\pi(\theta)$, and $\bar{g}*\bar{h}:=\pi[g*h]$; clearly $\bar{g}*\bar{h}$ is generic over $H_W$.  Then by Theorem \ref{thm_Sealing}, $H_W[\bar{g}*\bar{h}] \models $ ``$(H_{\bar{\theta}})^{H_W}$ is indestructibly guessing."  This statement is upward absolute to $V$ (since $V$ is an outer model of $H_W[\bar{g}*\bar{h}]$ with the same $\omega_1$).  Finally, notice that $(H_{\bar{\theta}})^{H_W} = \pi[W \cap H_\theta]$.  So the transitive collapse of $W \cap H_\theta$ is indestructibly guessing, which implies that $W \cap H_\theta$ is indestructibly guessing.
\end{proof}

Given a partial order $\mathbb{P}$, a regular $\theta \ge |\mathbb{P}|^+$, and a condition $p \in \mathbb{P}$, we say that $\boldsymbol{p}$ \textbf{is an $\boldsymbol{(M,\mathbb{P})}$-strong master condition} iff there is a condition $p|M \in M \cap \mathbb{P}$ such that for every $r \in M$ that is stronger than $p|M$, $r$ is compatible in $\mathbb{P}$ with $p$ (the condition $p|M$ is not typically unique, and is called a \textbf{reduction of $\boldsymbol{p}$ into $\boldsymbol{M}$}).  A partial order $\mathbb{P}$ is \textbf{strongly proper on a stationary set} iff there is a stationary $S \subseteq [H_{|\mathbb{P}|^+}]^\omega$ such that for all but nonstationarily many $M \in S$ and every $p \in M$, there is a $p' \le p$ such that $p'$ is an $(M,\mathbb{P})$-strong master condition.  This concept was isolated by Mitchell, and its main use is:
\begin{fact}[Mitchell~\cite{MR2452816}]\label{fact_Mitchell}
If $\mathbb{P}$ is strongly proper on a stationary set, then it has the $\omega_1$ covering and approximation properties.
\end{fact}

There are two kinds of posets that are strongly proper on a stationary set that we will use in this paper:  adding a Cohen real, and the following club-shooting poset, which is the ``continuous" version of Todorcevic's $\in$-collapse.  This poset is a special case of Neeman's ``decorated" poset from \cite{MR3201836}.  The role of the $f$ in Definition \ref{def_DecoratedPoset} is just to ensure that the $\in$-chain of elementary submodels added by the first coordinate is $\subseteq$-continuous.\footnote{Neeman's forcing is designed to preserve $\omega_1$ and $\theta$; we do not need preservation of $\theta$ for our purposes.}
\begin{definition}\label{def_DecoratedPoset}
Let $\theta \ge \omega_2$ be regular, and $X$ a stationary subset of $[H_\theta]^\omega$, which we will without loss of generality assume consists only of elementary submodels of $(H_\theta,\in)$.  The poset $\mathbb{C}^{\text{fin}}_{\text{dec}}(X)$ is the collection of all pairs of the form $(\mathcal{M}, f)$ where:
\begin{itemize}
 \item $\mathcal{M}$ is a finite subset of $X$, and for every $M, N \in \mathcal{M}$ with $M \ne N$, either $M \in N$ or $N \in M$.  We let $M_0, M_1, M_2, \dots, M_k$ be the unique enumeration of $\mathcal{M}$ such that $M_i \in M_{i+1}$ for every $i < k$.
 \item $f$ is a function from $\mathcal{M} \to H_\theta$, and has the property that $f(M_i) \in M_{i+1}$ for every $i < k$. 
\end{itemize}

The ordering is defined by:  $(\mathcal{N},h) \le (\mathcal{M},f)$ iff $\mathcal{N} \supseteq \mathcal{M}$ and $h(M) \supseteq f(M)$ for every $M \in \mathcal{M}$.
\end{definition}

\begin{fact}[similar to Neeman~\cite{MR3201836}, Section 4]\label{fact_DecoratedPosetFacts}
Suppose $X$ is a stationary subset of $[H_\theta]^\omega$ where $\theta \ge \omega_2$ is regular.  Then:
\begin{enumerate}
 \item  $\mathbb{C}^{\text{fin}}_{\text{dec}}(X)$ is strongly proper on a stationary set.  In particular, if $N \prec H_{(2^\theta)^+}$ is countable, $N \cap H_\theta \in X$, and $(\mathcal{M},f)$ is a condition in $N$, then $(\mathcal{M} \cup \{ N \cap H_\theta \}, f)$ is an $(N,\mathbb{C}^{\text{fin}}_{\text{dec}}(X))$-strong master condition.  
 \item If $G$ is generic over $V$ for $\mathbb{C}^{\text{fin}}_{\text{dec}}(X)$, then $\bigcup G$ is an $\in$-increasing filtration of $H^V_\theta$ of length $\omega_1$ consisting entirely of members of $X$.   Moreover, if $X_0 \in V$ was a stationary subset of $X$ and $\langle Q_i \ : \ i < \omega_1 \rangle \in V[G]$ is the generic filtration, then there are stationarily many $i < \omega_1$ such that $Q_i \in X_0$.

\end{enumerate}

In particular, $\mathbb{C}^{\text{fin}}_{\text{dec}}(X)$ has the $\omega_1$ covering and approximation properties, and forces $|H^V_\theta|=\omega_1$.
\end{fact}

We will also need:
\begin{lemma}\label{lem_ClubThruVisProper}
If $\mathbb{P}$ is proper and $\theta$ is a regular uncountable cardinal, then $\mathbb{P}*\dot{\mathbb{C}}^{\text{fin}}_{\text{dec}}\big( V \cap [\theta]^\omega \big)$ is proper.
\end{lemma}
\begin{proof}
Fix a regular $\Omega$ with $\mathbb{P},\theta \in H_\Omega$.  Let $N \prec (H_\Omega,\in, \theta, \mathbb{P})$ be countable, and $\big( p, (\dot{\mathcal{M}}, \dot{f}) \big) \in N$ be a condition in the two-step iteration.  Since $\mathbb{P}$ is proper, there is a $p' \le p$ that is an $(N,\mathbb{P})$-master condition, so in particular $p'$ forces $ N[\dot{G}_{\mathbb{P}}] \cap \theta = N \cap \theta \in V \cap [\theta]^\omega$.  Then by Fact \ref{fact_DecoratedPosetFacts}, $p'$ forces that $(\dot{\mathcal{M}} \cup \{ N \cap \theta \}, \dot{f})$ is an $(N[\dot{G}_{\mathbb{P}}], \dot{\mathbb{C}}^{\text{fin}}_{\text{dec}}(V \cap [\theta]^\omega))$ (strong) master condition. Then $\big( p', (\dot{\mathcal{M}} \cup \{ N \cap \theta \}, \dot{f}) \big)$ is a master condition for $N$ that is stronger than $\big( p, (\dot{\mathcal{M}}, \dot{f}) \big) $.
\end{proof}

\subsection{A theorem of Gitik and Velickovic}

For any set $R$ and any $T \subseteq \omega_1$, $R \searrow T:= \{ z \in R \ : \ z \cap \omega_1 \in T \}$.  A set $R \subset [H_\theta]^\omega$ is called \textbf{projective stationary} iff $R \searrow T$ is stationary for every stationary $T \subseteq \omega_1$ (this concept was isolated in Feng-Jech~\cite{MR1668171}).  We make heavy use of the following theorem of Velickovic, which slightly improved an earlier theorem of Gitik~\cite{MR820120}:

\begin{theorem}[Velickovic~\cite{MR1174395}, Lemma 3.15]\label{thm_Velick}
Suppose $V \subset W$ are transitive ZFC models, $\mathbb{R}^V \ne \mathbb{R}^W$, and that $\lambda$ is a $W$-regular cardinal with $\lambda \ge \omega_2^W$.  Then in $W$, for every stationary $R \subseteq \lambda \cap \text{cof}(\omega)$ and every stationary $T \subseteq \omega_1$, there are stationarily many $z \in [\lambda]^\omega \setminus V$ such that $z \cap \omega_1 \in T$ and $\text{sup}(z) \in R$.  In particular, $[\lambda]^\omega \setminus V$ is projective stationary.
\end{theorem}

The statement of Theorem \ref{thm_Velick} is slightly stronger than the version in \cite{MR1174395}, which didn't mention the set $R$ (just the projective stationarity of $[\lambda]^\omega \setminus V$).  But a close examination of Velickovic's proof shows the statement above, because he proved that (in $W$) given any stationary $T \subseteq \omega_1$ and any function $F: [\lambda]^{<\omega} \to \lambda$, for all but nonstationarily many members of the set
\[
\{ X \subset H_{(2^\lambda)} \ : \ X \cap \lambda \in \lambda \cap \text{cof}(\omega) \},
\]
there is a $z \in [X \cap \lambda]^\omega$ such that $z$ is closed under $F$, $\text{sup}(z) = X \cap \lambda$, $z \notin V$, and $z \cap \omega_1 \in T$.   This will be the case for any $X$ in the displayed set that has, as an element, Player II's winning strategy in the game from page 272 of \cite{MR1174395}.  In particular, this is true for some $X$ such that $X \cap \lambda \in R$.

We also will use:
\begin{lemma}\label{lem_LiftStatToStatOrd}
Suppose $\lambda$ is an uncountable ordinal, $S$ is a stationary subset of $[\lambda]^\omega$, $\theta$ is a regular uncountable cardinal strictly larger than $\lambda$, and $R$ is a stationary subset of $\theta \cap \text{cof}(\omega)$.  Then
\[
\left\{ u \in [\theta]^\omega \ : \ \text{sup}(u) \in R \text{ and } u \cap \lambda \in S \right\}
\]
is stationary in $[\theta]^\omega$.
\end{lemma}
\begin{proof}
Let $\mathfrak{A} = (\theta,\dots)$ be any first order structure on $\theta$ in a countable language; we need to find a countable elementary substructure of $\mathfrak{A}$ whose supremum is in $R$ and whose intersection with $\lambda$ is in $S$. Fix an $M \prec \mathfrak{A}$ of size $<\theta$ such that $\lambda \subseteq M$ and $M \cap \theta \in R$.  This is possible because $\lambda < \theta$ and $R$ is a stationary subset of $\theta \cap \text{cof}(\omega)$.  Fix a Skolemized structure $
\mathfrak{B}$ with universe $M$ in a countable language extending $\mathfrak{A} \restriction M$ such that any elementary substructure of $\mathfrak{B}$ is cofinal in $M \cap \theta$; this is possible because $M \cap \theta$ is $\omega$-cofinal.  Since $S$ is stationary in $[\lambda]^\omega$ and $\lambda \subseteq M$, Fodor's lemma implies that there is a $z \in S$ such that
\[
\widetilde{z} \cap \lambda = z.
\]
where $\widetilde{z}$ is the $\mathfrak{B}$-Skolem hull of $z$.  Then $\widetilde{z}$ is elementary in $\mathfrak{B}$ and hence $\text{sup}(\widetilde{z})  = M \cap \theta \in R$.  Since $\widetilde{z} \prec \mathfrak{B}$, $\mathfrak{B}$ extends $\mathfrak{A} \restriction M$, and $M \prec \mathfrak{A}$, then $\widetilde{z} \prec \mathfrak{A}$.  And $\widetilde{z} \cap \lambda = z \in S$.  

\end{proof}

\section{Controlling the internal part of a guessing model}\label{sec_ControlInternal}

In this section we prove some facts that will be used in most of the proofs of the paper.  The tight control over the internal and external parts is mainly needed for Theorem \ref{thm_Cox_DontNeedPlus2}.  A poset is called \textbf{$\boldsymbol{\omega_1}$-SSP} if it preserves all stationary subsets of $\omega_1$.  Given uncountable sets $H \subset H'$ and a subset $S \subseteq [H]^\omega$, $\boldsymbol{\textbf{Lift}^{H'}(S)}$ denotes the set $\{ z \in [H']^\omega \ : \ z \cap H \in S\}$; it is a standard fact that if $S$ is stationary, then so is its lifting.

\begin{theorem}\label{thm_TwoPosets}
Fix a (possibly nonstationary) subset $T$ of $\omega_1$, and a regular $\theta \ge \omega_2$.  
\begin{enumerate}
 \item\label{item_SemiproperPoset} There is an $\omega_1$-SSP poset $\mathbb{Q}^{\omega_1-\text{SSP}}_{T,\theta}$ that forces the following for every $V$-regular cardinal $\mu \in [\omega^V_2,\theta]$: $H^V_\mu$ has size $\omega_1$, is indestructibly guessing, internally unbounded, and its internal part is \textbf{exactly} (mod $\text{NS}_{\omega_1})$ the set $T$.  Also, every stationary subset of $\theta \cap \text{cof}(\omega)$ from the ground model remains stationary in the extension.

 \item\label{item_ProperPoset} There is a proper poset $\mathbb{Q}^{\text{proper}}_{T, \theta}$ that forces the following for every $V$-regular cardinal $\mu \in [\omega^V_2,\theta]$: $H^V_\mu$ has size $\omega_1$, is indestructibly guessing, and its internal part contains $T$.  Moreover, if $T$ was costationary in $V$, then both the internal and external parts of $H^V_\mu$ are forced to have stationary intersection with $T^c$.
 
 \textbf{In particular:}  if $T$ was stationary and costationary, then $\mathbb{Q}^{\text{proper}}_{T, \theta}$ forces ``$\text{int}(H^V_\mu)$ contains the stationary set $T$, and $\text{ext}(H^V_\mu)$ is a stationary subset of $T^c$" for every $V$-regular $\mu \in [\omega_2^V,\theta]$. 
 
\end{enumerate}
\end{theorem}

\begin{proof}

For part \ref{item_SemiproperPoset}:  $\mathbb{Q}^{\omega_1-\text{SSP}}_{T,\theta}$ is defined to be the poset

\begin{equation}\label{eq_3stepPoset}
\text{Add}(\omega) \ * \ \dot{\mathbb{C}}^{\text{fin}}_{\text{dec}} (\dot{X}^\theta_T) \ * \ \dot{\mathbb{S}}(H^V_\theta)
\end{equation}
where, letting $\dot{\sigma}$ be the $\text{Add}(\omega)$-name for its generic object,
\begin{eqnarray*}
\dot{X}^\theta_T:= & \text{Lift}^{H^V_\theta[\dot{\sigma}]} \Big( \big([H^V_\theta]^\omega  \searrow T \big)^V \ \cup \ \big( [\omega_2]^\omega \searrow T^c \big) \setminus V \Big) \\
= & \big\{ z \in \big[H_\theta[\dot{\sigma}] \big]^\omega \ : \ (z \cap H^V_\theta \in V \text{ and } z \cap \omega_1 \in T ) \text{ or } (z \cap \omega_2 \notin V \text{ and } z \cap \omega_1 \in T^c)   \big\},  
\end{eqnarray*}

$\dot{\mathbb{C}}^{\text{fin}}_{\text{dec}}(\dot{X}^\theta_T)$ is the poset from Definition \ref{def_DecoratedPoset}, and $\mathbb{S}(H^V_\theta)$ is the poset given by Theorem \ref{thm_Sealing}, assuming that the assumptions of that theorem hold, which we verify first.  For Theorem \ref{thm_Sealing} to be applicable, we need to check that the first two steps of \eqref{eq_3stepPoset} force $H^V_\theta \in \text{IU}$.  This is part of the following claim:

\begin{nonGlobalClaim}\label{clm_First2steps}
The first two steps of the poset  \eqref{eq_3stepPoset} force that 
\begin{enumerate}[label=(\alph*)]
\item\label{item_GuessingAndIU} $H^V_\theta \in G_{\omega_1} \cap \text{IU}$;

\item\label{item_PreserveStatTheta} all $V$-stationary subsets of $\theta \cap \text{cof}(\omega)$ remain stationary; 

\item\label{item_PreserveStatOmega1} all $V$-stationary subsets of $\omega_1$ remain stationary; and

\item\label{item_InternalPartIsT} for all $\mu \in [\omega_2^V,\theta]$ that are regular in $V$, $H^V_\mu$ has internal part exactly $T$.
\end{enumerate}
\end{nonGlobalClaim}
\begin{proof}
(of Claim \ref{clm_First2steps}):  the proofs of parts \ref{item_GuessingAndIU} and \ref{item_PreserveStatTheta} break into cases, depending on whether or not $T$ is stationary in $V$.  Though the proofs of these parts could be combined, we find it conceptually simpler to prove each part separately, at the expense of some slight repetition. 

For part \ref{item_GuessingAndIU}:  the poset $\text{Add}(\omega)$ is strongly proper, and the 2nd step is forced to be strongly proper on a stationary set by Fact \ref{fact_DecoratedPosetFacts}, \emph{provided that} the set $\dot{X}_T^\theta$ is forced by $\text{Add}(\omega)$ to be a stationary subset of $\big[  H_\theta[\dot{\sigma}] \big]^\omega$.  We verify this now.  Let $\sigma$ be $\text{Add}(\omega)$-generic; we break into cases depending on whether $T$ is stationary in $V$.  If $T$ is stationary in $V$, then $[H_\theta]^\omega \searrow T$ is stationary, so properness of $\text{Add}(\omega)$ implies that the set 
\[
\big( [H^V_\theta]^\omega \searrow T \big)^V,
\]
remains stationary in $V[\sigma]$, and hence its lifting to $H^V_\theta[\sigma]$---which is a subset of $(\dot{X}^\theta_T)_\sigma$---is stationary from the point of view of $V[\sigma]$.  On the other hand, if $T$ is nonstationary in $V$, then $T^c$ contains a club subset of $\omega_1$.  By Gitik-Velickovic's Theorem \ref{thm_Velick} (noting that $\omega_2^V = \omega_2^{V[\sigma]}$), the set $\big( [\omega_2]^\omega \big) \setminus V$ is (projective) stationary in $V[\sigma]$.  So in particular $\big( [\omega_2]^\omega  \searrow T^c \big) \setminus V$ is stationary, and hence its lifting to $ H^V_\theta[\sigma]$ is stationary.  And this lifting is a subset of $(\dot{X}^\theta_T)\sigma$.  

In summary, whether $T$ is stationary or not, the set $\dot{X}^\theta_T$ is forced by $\text{Add}(\omega)$ to be stationary, and hence the first two steps of \eqref{eq_3stepPoset} are of the form ``strongly proper, followed by strongly proper on a stationary set".  So this 2-step iteration has the $\omega_1$ approximation and cover properties by Facts \ref{fact_Mitchell} and \ref{fact_CoverApproxClosed2step}.  This, in turn, ensures that $H^V_\mu$ will be $\omega_1$-guessing and internally unbounded after the first two steps of the iteration \eqref{eq_3stepPoset}, for all $V$-regular $\mu \in [\omega_2,\theta]$.  This completes the proof of part \ref{item_GuessingAndIU} of the claim.

Next we verify part \ref{item_PreserveStatTheta} of the claim.  Let $R$ be a stationary subset of $\theta \cap \text{cof}(\omega)$ from the ground model.  Consider two cases, depending on whether or not $T$ is stationary in $V$.  If $T$ is stationary in $V$, then 
\[
V \models \ \widetilde{R}^V_T:=\left\{ z \in \big([H_\theta]^\omega \big)^V \ : \ \text{sup}(z) \in R  \text{ and } z \cap \omega_1 \in T \right\} \text{ is stationary in } [\theta]^\omega.
\]
By properness of $\text{Add}(\omega)$, $\widetilde{R}^V_T$ remains stationary in $V[\sigma]$, and the lifting of $\widetilde{R}^V_T$ to $H_\theta[\sigma]$ is a stationary subset of $(\dot{X}^\theta_T)_\sigma$.  By Fact \ref{fact_DecoratedPosetFacts}, the second step of the forcing \eqref{eq_3stepPoset} preserves the stationarity of $\widetilde{R}^V_T$.  Since the supremum of every element of $\widetilde{R}^V_T$ lies in $R$, it follows that $R$ is still stationary as well.

Now suppose $T$ is nonstationary in $V$.  By Fact \ref{fact_DecoratedPosetFacts} and the definition of $\dot{X}^\theta_T$, it suffices to prove that 
\[
V[\sigma] \models \left\{ z \in [\theta]^\omega \ : \ \text{sup}(z) \in R, \ z \cap \omega_2 \notin V \text{ and } z \cap \omega_1 \in T^c \right\} \text{ is stationary in } [\theta]^\omega.
\]
Now $\omega_2^V = \omega_2^{V[\sigma]}$, and $T^c$ contains a club subset of $\omega_1$ by our case.  So, in turn, it suffices to show that 
\begin{equation}\label{eq_ToShowInVsigma}
V[\sigma] \models \left\{ z \in [\theta]^\omega \ : \  \text{sup}(z) \in R \text{ and } z \cap \omega_2 \in [\omega_2]^\omega \setminus  V \right\} \text{ is stationary in } [\theta]^\omega.
\end{equation}
Since $\omega_2^V = \omega_2^{V[\sigma]}$, Theorem \ref{thm_Velick} implies that $[\omega_2]^\omega \setminus V$ is stationary in $V[\sigma]$ (where the $\lambda$ from the statement of that theorem is taken to be $\omega_2$).  By properness of $\text{Add}(\omega)$, $R$ is still a stationary subset of $\theta \cap \text{cof}(\omega)$ in $V[\sigma]$.  If $\theta = \omega_2$ then  \eqref{eq_ToShowInVsigma} follows from Theorem \ref{thm_Velick} directly; if $\theta > \omega_2$ then \eqref{eq_ToShowInVsigma} follows from Theorem \ref{thm_Velick} (again with $\lambda = \omega_2$) together with Lemma \ref{lem_LiftStatToStatOrd}.

Next we verify part \ref{item_PreserveStatOmega1} of the claim.  Let $S$ be a stationary subset of $\omega_1$, and let $\sigma$ be $(V,\text{Add}(\omega))$-generic.  Then at least one of $S \cap T$ or $S \cap T^c$ is stationary.  Suppose first that $S \cap T$ is stationary.  Then $([H^V_\theta]^\omega)^V  \searrow (S \cap T)$ is stationary in $V$, and remains so in $V[\sigma]$ by properness of $\text{Add}(\omega)$; then its lifting to $H^V_\theta[\sigma]$ is a stationary subset of of $X^\theta_T:=(\dot{X}^\theta_T)_\sigma$.  So by Fact \ref{fact_DecoratedPosetFacts} (viewing $V[\sigma]$ as the ground model), $\mathbb{C}^{\text{fin}}_{\text{dec}}(X^\theta_T)$ preserves the stationarity of $\text{Lift}^{H^V_\theta[\sigma]}\Big(\big([H^V_\theta]^\omega\big)^V  \searrow (S \cap T)\Big)$, and in particular the stationarity of $S \cap T$.  Now suppose $S \cap T^c$ is stationary.  Then it is still stationary in $V[\sigma]$, so by Gitik-Velickovic's Theorem \ref{thm_Velick}, $\Big([\omega_2]^\omega \searrow (S \cap T^c) \Big) \setminus V$ is stationary, and hence its lifting to $H^V_\theta[\sigma]$ is stationary.  Moreover, this lifting is a subset of $X^\theta_T$.  Then by Fact \ref{fact_DecoratedPosetFacts} (again viewing $V[\sigma]$ as the ground model), the poset $\mathbb{C}^{\text{fin}}_{\text{dec}}(X^\theta_T)$ preserves stationarity of $\text{Lift}^\theta \Big( \big([\omega_2]^\omega \searrow (S \cap T^c) \big) \setminus V \Big)$, and hence stationarity of $S 
\cap T^c$.

Finally we verify part \ref{item_InternalPartIsT} of the claim.  The poset $\mathbb{C}(X^\theta_T)$ adds a filtration $\langle Q_i \ : \ i < \omega_1 \rangle$ of $H^V_\theta[\sigma]$ with the property that whenever $i$ is in the club $C := \{ i < \omega_1 \ : \ Q_i \cap \omega_1 = i \}$, then: 
\begin{itemize}
 \item if $i \in T$ then $Q_i \cap H^V_\theta \in V$, which implies that $Q_i \cap H^V_\mu \in V$ for all $\mu \in [\omega_2,\theta]$; and
 \item if $i \in T^c$ then $Q_i \cap \omega_2 \notin V$, which implies that $Q_i \cap H^V_\mu \notin V$ for all $\mu \in [\omega_2,\theta]$.
\end{itemize}
It follows that for each $V$-regular $\mu \in [\omega_2,\theta]$, the filtration $\langle Q_i \cap H^V_\mu \ : \ i < \omega_1 \rangle$ has the property that for every $i \in C$, $Q_i \cap H_\mu^V$ is an element of $H^V_\mu$ if and only if $i \in T$.  This shows that the internal part of $H^V_\mu$ is forced to be exactly $T$.

\end{proof}

Then by Claim \ref{clm_First2steps} and Theorem \ref{thm_Sealing}, the third step $\dot{\mathbb{S}}(H^V_\theta)$ of the poset \eqref{eq_3stepPoset} forces $H^V_\theta$ to be \emph{indestructibly} guessing (and still internally unbounded).  Also, Claim \ref{clm_First2steps} and Lemma \ref{lem_InternalPartAbsolute} imply that for every $V$-regular $\mu \in [\omega_2^V,\theta]$, $H^V_\mu$ still has internal part exactly $T$ after forcing with $\dot{\mathbb{S}}(H^V_\theta)$, since the latter preserves $\omega_1$ (it is c.c.c.).  Since $\dot{\mathbb{S}}(H^V_\theta)$ is forced to be c.c.c., in particular it preserves all stationary subsets of $\omega_1$ and of $\theta \cap \text{cof}(\omega)$ that lie in $V^{\text{Add}(\omega) \ * \ \dot{\mathbb{C}}^{\text{fin}}_{\text{dec}} (\dot{X}^\theta_T)}$.  Together with Claim \ref{clm_First2steps}, it follows that all stationary subsets of $\theta \cap \text{cof}(\omega)$, and of $\omega_1$, from $V$ are preserved by the 3-step iteration \eqref{eq_3stepPoset}.  This completes the proof of part \ref{item_SemiproperPoset} of the theorem.

Part \ref{item_ProperPoset} is similar:    $\mathbb{Q}^{\text{proper}}_{T,\theta}$ is defined similarly to the poset $\mathbb{Q}^{\omega_1\text{-SSP}}_{T,\theta}$ defined in \eqref{eq_3stepPoset}, except that instead of the $\text{Add}(\omega)$-name $\dot{X}^\theta_T$, we use the $\text{Add}(\omega)$-name
\begin{eqnarray*}
\dot{Y}^\theta_T:= &  \text{Lift}^{H^V_\theta[\sigma]}\Big( ([H^V_\theta]^\omega)^V \ \cup \   \big( [\omega_2]^\omega \searrow T^c \big) \setminus V  \Big) \\
 & = \big\{ z \in \big[H_\theta[\sigma] \big]^\omega \ : \ \big( z \cap H^V_\theta \in V  \big) \text{ or } \big( z \cap \omega_1 \in T^c \text{ and } z \cap \omega_2 \notin V  \big)  \big\}.
\end{eqnarray*}

The proof is similar to the proof of part \ref{item_SemiproperPoset}, so we only briefly sketch it.  Roughly, the fact that $\dot{Y}^\theta_T$ contains the lifting of \emph{all} of $([H^V_\theta]^\omega)^V$---rather than just of $([H^V_\theta]^\omega)^V \searrow T$ as was the case with the club shooting portion of  $\mathbb{Q}^{\omega_1\text{-SSP}}_{T,\theta}$---ensures that the iteration $\mathbb{Q}^{\text{proper}}_{T,\theta}$ will indeed be proper.  Now suppose $\sigma$ is $(V,\text{Add}(\omega))$-generic and $Y^\theta_T:=(\dot{Y}^\theta_T)_\sigma$.  Let $\langle Q_i \ : \ i < \omega_1 \rangle$ be a generic filtration of $H^V_\theta[\sigma]$ added by $\mathbb{C}^{\text{fin}}_{\text{dec}}(Y^\theta_T)$ over $V[\sigma]$.  By definition of $Y^\theta_T$, if $i$ is such that $Q_i \cap \omega_1 = i$ and $i \in T$, then $Q_i \cap H^V_\theta \in V$, and hence $Q_i \cap H^V_\mu \in V$ for every $V$-regular $\mu \in [\omega_2^V,\theta]$.  It follows that the internal part of any such $H^V_\mu$ contains (mod $\text{NS}_{\omega_1}$) the set $T$.  In particular, if $T$ is stationary, then $H^V_\mu \in \text{IS}^{V[\sigma*\vec{Q}]}$.  Now for the ``moreover" clause of part \ref{item_ProperPoset}, assume that $T$ is costationary; we want to show that both the internal and external part of each $H^V_\mu$ have stationary intersection with $T^c$.  That the internal part of each $H^V_\mu$ has stationary intersection with $T^c$ follows from Fact \ref{fact_DecoratedPosetFacts} together with the fact that $([H^V_\theta]^\omega)^V \searrow T^c$ is, in $V[\sigma]$, a stationary subset of $Y^\theta_T$ (more precisely, its lifting to $H_\theta[\sigma]$ is a stationary subset of $Y^\theta_T$).  For the external part, the Gitik-Velickovic Theorem \ref{thm_Velick} ensures that, in $V[\sigma]$, the lifting of $\big( [\omega_2]^\omega \searrow T^c \big) \setminus V$ is a stationary subset of $Y^\theta_T$, and hence by Fact \ref{fact_DecoratedPosetFacts} there are stationarily many $i < \omega_1$ such that $Q_i \cap \omega_1 =i \in T^c$ and $Q_i \cap \omega_2^V \notin V$.  For such $i$, $Q_i \cap H^V_\mu \notin H^V_\mu$ for every $V$-regular $\mu \in [\omega_2^V,\theta]$.

The proof that $H^V_\theta$ is forced by $\mathbb{Q}^{\text{proper}}_{T,\theta}$ to be indestructibly guessing is almost identical to the proof from part \ref{item_SemiproperPoset}.
\end{proof}

\begin{corollary}\label{cor_ControlInternal}
Assume $\theta \ge \omega_2$ is regular and $T \subseteq \omega_1$.  Let $\mathbb{Q}^{\text{proper}}_{T, \theta}$ and $\mathbb{Q}^{\omega_1-\text{SSP}}_{T,\theta}$ be the posets given by Theorem \ref{thm_TwoPosets}.  
\begin{enumerate}
 \item\label{item_SemiProperVersion_model} Assume $W \prec (H_{(2^\theta)^+}, \in, T)$, $|W|=\omega_1 \subset W$, and there exists a $(W,\mathbb{Q}^{\omega_1-\text{SSP}}_{T,\theta})$-generic filter.  Then for every regular $\mu \in [\omega_2,\theta] \cap W$, $W \cap H_\mu$ is indestructibly guessing, has internal part exactly $T$, and external part exactly $T^c$.  In particular, if $T$ is nonstationary then $W \cap H_\mu \in \text{GIU} \setminus \text{IS}$, and if $T$ is stationary and costationary then $W \cap H_\mu \in \text{GIS} \setminus \text{IC}$.

 If the $(W,\mathbb{Q}^{\omega_1-\text{SSP}}_{T,\theta})$-generic filter is also stationary correct (with respect to names for stationary subsets of $\omega_1$ that lie in $W$), then for every $R \in W$ that is a stationary subset of $\theta \cap \text{cof}(\omega)$, $R \cap \text{sup}(W \cap \theta)$ is stationary in $\text{sup}(W \cap \theta)$.

 \item\label{item_ProperVersion_model}  Let  $\dot{S}^{\omega_2}_{\text{external}}$ be the $\mathbb{Q}^{\text{proper}}_{T, \theta}$-name for the external part of $H^V_{\omega_2}$.   If $T$ is costationary, then $\dot{S}^{\omega_2}_{\text{external}}$ is forced to be a stationary subset of $T^c$.  Furthermore, if $W \prec (H_{(2^\theta)^+}, \in, T)$, $|W|=\omega_1 \subset W$, and there exists a $(W, \mathbb{Q}^{\text{proper}}_{T, \theta})$-generic filter that interprets $\dot{S}^{\omega_2}_{\text{external}}$ as a stationary set, then for every regular $\mu \in [\omega_2,\theta] \cap W$, $W \cap H_\mu$ is indestructibly guessing, its internal part contains $T$, and its external part is stationary and costationary in $T^c$.  In particular, if $T$ is stationary and costationary, then $W \cap H_\mu \in \text{G}_{\omega_1} \cap \text{IS} \setminus \text{IC}$ for every regular $\mu \in [\omega_2,\theta] \cap W$.

\end{enumerate}
\end{corollary}
\begin{proof}

For part \ref{item_ProperVersion_model}, the fact that $\dot{S}^{\omega_2}_{\text{external}}$ is forced to be stationary follows immediately from Theorem \ref{thm_TwoPosets}.  Also, Observation \ref{obs_InternalPartProjects} implies that 
\begin{equation}\label{eq_ExternalPartGrows}
\Vdash \ \ \forall \mu \in \text{REG} \cap [\omega_2,\theta] \ \  \ \text{ext}(H^V_\mu) \supseteq \text{ext}(H^V_{\omega_2}) =  \dot{S}^{\omega_2}_{\text{external}}. 
\end{equation}
Now suppose $W$ is as in the statement of part \ref{item_ProperPoset}, and that $g$ is a  $(W,\mathbb{Q}^{\text{proper}}_{T, \theta})$-generic filter such that $(\dot{S}^{\omega_2}_{\text{external}})_g$ is a stationary subset of $\omega_1$ (in $V$).  Let $\sigma_W:H_W \to W \prec (H_{(2^\theta)^+}, \in, T) $ be the inverse of the transitive collapse of $W$, and $b_W := \sigma^{-1}_W(b)$ for every $b \in W$.  Let $\bar{g}$ be the pointwise image of $g$ under the collapsing map.  Note that $T \in W$ and $T$ is not moved by the collapsing map.  It can also be easily verified that $(\dot{S}^{\omega_2}_{\text{external}})_g$ is the same as $\big(\sigma_W^{-1}(\dot{S}^{\omega_2}_{\text{external}})\big)_{\bar{g}}$; let $S$ denote this set.  
Fix any regular $\mu \in W \cap [\omega_2,\theta]$. Then by \eqref{eq_ExternalPartGrows} and Theorem \ref{thm_TwoPosets}, $H_W[\bar{g}] \models$ ``$\sigma^{-1}(H_\mu)$ is indestructibly guessing, size $\omega_1$, its internal part contains $\sigma^{-1}(T) = T$, and its external part contains $S$".  Since $H_W[\bar{g}]$ is a transitive $\text{ZF}^-$ model and $V$ is an outer model of $H_W[\bar{g}]$ with the same $\omega_1$, those statements are upward absolute to $V$.  Finally, notice that the statement ``$X$ is indestructibly guessing, its internal part contains $T$, and its external part contains $S$"---viewing $T$ and $S$ as fixed parameters---is invariant across sets that are $\in$-isomorphic to $X$.  In particular, $V$ believes the same statement about $\sigma_W [ \sigma_W^{-1}(H_\mu)]$, which is just $W \cap H_\mu$.  This completes the proof of part \ref{item_ProperVersion_model}.

The proof of \ref{item_SemiProperVersion_model} is similar, using instead part \ref{item_SemiproperPoset} of Theorem \ref{thm_TwoPosets}.  The only new thing to check is that if the $W$-generic filter $g$ is stationary correct, then $W$ is diagonally reflecting for stationary subsets of $\theta \cap \text{cof}(\omega)$ lying in $W$.  To see this, note that $\theta$ is collapsed to $\omega_1$ and still has uncountable cofinality.  Let $\dot{f}$ be a $\mathbb{Q}^{\omega_1-\text{SSP}}_{T,\theta}$-name for a strictly increasing, continuous, cofinal function from $\omega_1 \to \theta$.  Suppose $R \in W$ and $R$ is stationary in $\theta \cap \text{cof}(\omega)$.  Let $\dot{T}_R$ be the name for $\{ i < \omega_1 \ : \ \dot{f}(i) \in \check{R} \}$.  Since $R$ remains stationary by Theorem \ref{thm_TwoPosets}, it follows that $\dot{T}_R$ is forced to be a stationary subset of $\omega_1$, and hence its interpretation by $g$ is really stationary in $\omega_1$.   By Viale-Weiss~\cite{VW_ISP}, $W \cap \theta$ is an $\omega$-closed set of ordinals.  This, together with $W$-genericity of $g$, implies that $(\dot{f})_g$ maps $\omega_1$ continuously and cofinally into $\text{sup}(W \cap \theta)$.  Then the pointwise image of $(\dot{T}_R)_g$ under $(\dot{f})_g$ is a stationary subset  of $\text{sup}(W \cap \theta)$, and also contained in $R$.
\end{proof}

\section{Preservation of forcing axioms via generic embeddings}\label{sec_PreservViaEmbed}

The material in this section is motivated by Woodin~\cite{MR2723878} and Viale (\cite{Viale_GuessingModel},  \cite{Viale_ChangModel}), where stationary tower forcing is used in conjunction with forcing axioms.   Unlike those settings, which made use of large cardinals in the universe to ensure \emph{wellfounded} generic ultrapowers via stationary tower forcing, wellfoundedness will not be important for the goals of this article.  The main theorem in this section (Theorem \ref{thm_PreservationFA}) makes preservation arguments for forcing axioms very closely resemble arguments that involve lifting traditional large cardinal embeddings; this makes forcing axiom preservation arguments conceptually clearer, in the author's opinion.  Theorem \ref{thm_PreservationFA} is possibly folklore, but the author is not aware of any presentation of it in the literature.  The section concludes with several preservation results from the literature which can be viewed as instances of Theorem \ref{thm_PreservationFA}.

We make the following definition, which holds of all standard classes of posets (proper, c.c.c., etc.):
\begin{definition}\label{def_Nice}
We say that a class $\Gamma$ of posets is \textbf{closed under restrictions} iff whenever $\mathbb{Q} \in \Gamma$ and $q \in \mathbb{Q}$, then 
\[
\mathbb{Q} \restriction q := \{ p \in \mathbb{Q} \ : \ p \le q \}
\]
(with order inherited from $\mathbb{Q}$) is in $\Gamma$.
\end{definition}

Note that in the statement of Theorem \ref{thm_PreservationFA}, we do \textbf{not} require that $\Gamma$ (or $\dot{\Delta}$) is closed under taking regular suborders; this will be convenient in the proof of the theorems about stationary reflection, since e.g.\ ``being proper and and forcing $H^V_{\omega_2} \notin \text{IC}$" is not closed under regular suborders (the regular suborder will be proper, but may fail to force $H^V_{\omega_2} \notin \text{IC}$).  However, notice that if $\text{FA}(\Gamma)$ holds and $\widetilde{\Gamma}$ is the closure of $\Gamma$ under taking regular suborders, then $\text{FA}(\widetilde{\Gamma})$ also holds.  This is not necessarily true of the ``plus" versions, unless one also requires that the \emph{quotient} of the regular suborder is $\omega_1$-SSP.  We will use the following observation several times:

\begin{observation}\label{obs_FMS}
If $\text{FA}(\Gamma)$ holds and $\mathbb{P}$ is a regular suborder of some member of $\Gamma$ (but possibly $\mathbb{P} \notin \Gamma$), then $\mathbb{P}$ preserves stationary subsets of $\omega_1$.  This is because, by Foreman-Magidor-Shelah~\cite{MR924672}, every member of $\Gamma$ must be $\omega_1$-SSP, and this property clearly is inherited by regular suborders.
\end{observation}

The key to the preservation theorem is the following lemma of Woodin.

\begin{lemma}[Woodin~\cite{MR2723878}, proof of Theorem 2.53]\label{lem_WoodinKeyLemma}

If $\Gamma$ is a class of partial orders and $\mu \le \omega_1$ is a cardinal (possibly $\mu = 0$), the following are equivalent:
\begin{enumerate}
 \item $\text{FA}^{+\mu}(\Gamma)$;
 \item\label{item_WoodinLemmaStatCorrect} For every $\mathbb{P} \in \Gamma$, every $\mu$-sequence $\dot{\vec{S}}=\langle \dot{S}_i \ : \ i < \mu \rangle$ of $\mathbb{P}$-names for stationary subsets of $\omega_1$, and every (equivalently, some) regular $\theta > |\wp(\mathbb{P})|$, there are stationarily many $W \in \wp^*_{\omega_2}(H_\theta)$ such that there exists an $\dot{\vec{S}}$-correct, $(W,\mathbb{P})$-generic filter.
\end{enumerate}
\end{lemma}

\begin{remark}\label{rem_WoodinVariant}
Suppose $\mu = \omega_1$, and $\Gamma$ has the additional property that for every $\mathbb{P} \in \Gamma$ and every cardinal $\lambda$, there exists a $\mathbb{P}$-name $\dot{\mathbb{Q}}$ such that $\mathbb{P}*\dot{\mathbb{Q}} \in \Gamma$ and $\mathbb{P}*\dot{\mathbb{Q}} \Vdash |\lambda| \le \omega_1$.  This holds, for example, if $\Gamma$ is the class of proper posets, or $\omega_1$-SSP posets (but not c.c.c. posets).  Then clause \ref{item_WoodinLemmaStatCorrect} of Lemma \ref{lem_WoodinKeyLemma} can be replaced by: ``For every $\mathbb{P} \in \Gamma$ and every (equivalently, some) regular $\theta > |\wp(\mathbb{P})|$, there are stationarily many $W \in \wp^*_{\omega_2}(H_\theta)$ such that there exists a $(W,\mathbb{P})$-generic filter $g$, such that for every $\dot{S} \in W$ that names a stationary subset of $\omega_1$, the evaluation of $\dot{S}$ by $g$ is stationary (in $V$)."  
\end{remark}

Woodin's Lemma \ref{lem_WoodinKeyLemma}, together with basic theory of generic ultrapowers, yields the following Theorem \ref{thm_VariantWoodinTheorem}, which is a variant of Theorem 2.53 of Woodin~\cite{MR2723878} with weaker hypotheses (e.g.\ no Woodin cardinals are needed) and weaker conclusion (e.g.\ the generic embeddings constructed in Theorem \ref{thm_VariantWoodinTheorem} are not even necessarily wellfounded, though the ones from Woodin~\cite{MR2723878} are generic almost huge embeddings).  Essentially, $\text{FA}(\Gamma)$ can be characterized by existence of generically ``supercompact" elementary embeddings where the (possibly illfounded) target model has $V$-generics for the relevant forcing.

\begin{theorem}\label{thm_VariantWoodinTheorem}[minor variant of Theorem 2.53 of Woodin~\cite{MR2723878}]
Let $\Gamma$ be a class of posets closed under restrictions, and let $\mu \le \omega_1$ be a cardinal.  The following are equivalent:
\begin{enumerate}
 \item\label{item_FAchar_FA} $\text{FA}^{+\mu}(\Gamma)$;
 \item\label{item_FAchar_embedding} For every $\mathbb{Q} \in \Gamma$, every $q \in \mathbb{Q}$, every sequence $\langle \dot{S}_i \ : \ i < \mu \rangle$ of $\mathbb{Q}$-name for stationary subsets of $\omega_1$, and every (equivalently, some) regular $\theta > |\wp(\mathbb{Q})|$, there is a generic elementary embedding $j: V \to N$ such that:
 \begin{enumerate}
  \item $H^V_{\theta}$ is in the wellfounded part of $N$ (we assume the wellfouned part of $N$ has been transitivized), and $| H^V_{\theta}|^N=\omega_1$;

  \item $j  \restriction H^V_{\theta} \in N$
  \item $\text{crit}(j) = \omega_2^V$;
  \item\label{item_ClauseStatCorrect} There exists some $H \in N$ such that $q \in H$, $H$ is $(V,\mathbb{Q})$-generic, and for every $i < \mu$, $N \models$ ``$(\dot{S}_i)_H$ is a stationary subset of $\omega_1$".
\end{enumerate}   
\end{enumerate}
\end{theorem}

We make a remark that is parallel to Remark \ref{rem_WoodinVariant}:
\begin{remark}
Suppose $\mu = \omega_1$, and $\Gamma$ has the additional property that for every $\mathbb{Q} \in \Gamma$ and every cardinal $\lambda$, there exists a $\mathbb{Q}$-name $\dot{\mathbb{R}}$ such that $\mathbb{Q}*\dot{\mathbb{R}} \in \Gamma$ and $\mathbb{Q}*\dot{\mathbb{R}} \Vdash |\lambda| \le \omega_1$.  This holds, for example, if $\Gamma$ is the class of proper posets, or $\omega_1$-SSP posets (but not c.c.c. posets).  Then clause \ref{item_ClauseStatCorrect} of Theorem \ref{thm_VariantWoodinTheorem} can be replaced by: ``There exists some $H \in N$ such that $q \in H$, $H$ is $(V,\mathbb{Q})$-generic filter, and every stationary subset of $\omega_1$ in $V[H]$ remains stationary in $N$."
\end{remark}

\begin{proof}
(of Theorem \ref{thm_VariantWoodinTheorem}):  Since this is very similar to the proof of Theorem 2.53 of Woodin~\cite{MR2723878}, we only briefly sketch the proof.  

First assume $\text{FA}^{+\mu}(\Gamma)$.  Fix $\mathbb{Q} \in \Gamma$, $q \in \mathbb{Q}$, a sequence $\dot{\vec{S}}=\langle \dot{S}_i \ : \ i < \mu \rangle$ of $\mathbb{Q}$-names for stationary subsets of $\omega_1$, and a large regular $\theta$.   Since $\Gamma$ is closed under restrictions, $\mathbb{Q} \restriction q$ is an element of $\Gamma$, and hence the forcing axiom holds for it.  Let $R$ be the set of $W \in \wp_{\omega_2}(H_\theta)$ such that $\omega_1 \subset W$, $\dot{\vec{S}} \in W$, and there exists an $\dot{\vec{S}}$-correct, $(W,\mathbb{Q} \restriction q)$-generic filter $h_W$.  By Lemma \ref{lem_WoodinKeyLemma}, $R$ is stationary.  For each $W \in R$, let $\bar{W}$ be the transitive collapse of $W$, and $\sigma_W: \bar{W} \to W \prec H_\theta$ be the inverse of the Mostowski collapsing map.  Then $\bar{h}_W:=\sigma_W^{-1}[h_W]$ is an $(\bar{W}, \sigma_W^{-1}(\mathbb{Q}))$-generic filter, $\bar{q}_W :=\sigma_W^{-1}(q) \in \bar{h}_W$, $\text{crit}(\sigma_W) = \omega_2^{\bar{W}}$, and the evaluation of $\sigma_W^{-1}(\dot{S}_i)$ by $\bar{h}_W$ is a stationary subset of $\omega_1$ (in $V$), for all $i < \mu$. Let $\mathcal{J}$ be the restriction of the nonstationary ideal to $R$.  Let $U$ be $(V,\mathcal{J}^+)$-generic and $j: V \to_U N$ be the resulting generic elementary embedding.  Then standard applications of Los' Theorem (see Foreman~\cite{MattHandbook}) ensure that $j$ has the desired properties mentioned above. In particular, $[\text{id}_R]_U = j[H^V_\theta]$, $j \restriction H^V_\theta =[W \mapsto \sigma_W]_U$, $H^V_\theta=[W \mapsto \bar{W}]_U$ and is an element of the (transitivized) wellfounded part of $N$, $\text{crit}(j) = \omega_2^V$, and $H:=[ W \mapsto \bar{h}_W ]_U$ is an $(H^V_\theta,\mathbb{Q} \restriction q)$-generic filter such that from the point of view of the generic ultrapower $N$, $(\dot{S}_i)_H$ is a stationary subset of $\omega_1$ for all $i < \mu = j(\mu)$.  Since $\theta$ was chosen sufficiently large from the start, $H$ is also $V$-generic.

Now we prove the converse.  Assume \ref{item_FAchar_embedding}, and that $\mathbb{Q} \in \Gamma$.  Fix any $\mu$-sequence $\dot{\vec{S}} = \langle \dot{S}_i \ : \ i < \mu \rangle$ of $\mathbb{Q}$-names for stationary subsets of $\omega_1$.  By Woodin's Theorem \ref{thm_VariantWoodinTheorem} it suffices to show that if $F: [H_\theta]^{<\omega} \to H_\theta$ with $F \in V$, then there is a $W \in \wp_{\omega_2}(H_\theta)$ that is closed under $F$ such that $\omega_1 \subset W$ and there exists a $\dot{\vec{S}}$-correct, $(W,\mathbb{Q})$-generic, $\omega_1$-stationary correct filter.  Let $j: V \to N$ and $H \in N$ be as in assumption \ref{item_FAchar_embedding}.   Now $j[H^V_\theta]$ is closed under $j(F)$, and $j[H^V_\theta] \in N$ by assumption.  Also note that $j[H] \in N$, since $H \in N$, $j \restriction H^V_\theta  \in N$, and $\mathbb{Q} \in H^V_\theta$.  It is routine to check that $j[H^V_\theta]$ and $j[H]$ witness that $N \models$ ``there is a model containing $\omega_1$ and closed under $j(F)$ for which there exists a $j(\dot{\vec{S}})=\langle j(\dot{S}_i) \ : \ i < \mu = j(\mu) \rangle$-correct, generic filter for $j(\mathbb{Q})$".  Then elementarity of $j$ yields the analogous statement in $V$, completing the proof.
\end{proof}

We now prove Theorem \ref{thm_PreservationFA} (from page \pageref{thm_PreservationFA}), which is our main tool for preservation of forcing axioms.  It makes preservation of forcing axioms closely resemble arguments where large cardinal embeddings are lifted after a ``preparation on the $j$ side", where the preparation is designed to provide a master condition.  Both directions of Theorem \ref{thm_VariantWoodinTheorem} will be used in the proof of Theorem \ref{thm_PreservationFA}:  the forward direction of Theorem \ref{thm_VariantWoodinTheorem} (using the class $\Gamma$ in $V$) will be used to obtain a generic embedding $j: V \to N$ with the relevant properties.  We will generically extend the embedding $j$ to domain $V^{\mathbb{P}}$ (using the master condition provided by $\dot{\mathbb{R}}$), and then apply the backwards direction of Theorem \ref{thm_VariantWoodinTheorem} (using the class $\dot{\Delta}$ in $V^{\mathbb{P}}$) to ensure that $\text{FA}^{+\mu}(\dot{\Delta})$ holds in $V^{\mathbb{P}}$.

\begin{proof}

(of Theorem \ref{thm_PreservationFA}):  In order to show that $\text{FA}^{+\mu}(\dot{\Delta})$ holds in $V^{\mathbb{P}}$, it suffices to verify that clause \ref{item_FAchar_embedding} of Theorem \ref{thm_VariantWoodinTheorem} holds in $V^{\mathbb{P}}$ (with respect to the class $\dot{\Delta}$).  More precisely, we prove that $1_{\mathbb{P}}$ forces that ``whenever $\mathbb{Q}$ is a poset in $\dot{\Delta}$, $q \in \mathbb{Q}$, and $\langle \dot{S}_i \ : \ i < \mu \rangle$ is a $\mu$-sequence of $\mathbb{Q}$-names for stationary subsets of $\omega_1$, then there exists a generic elementary embedding 
\[
\pi: V[\dot{G}] \to M
\]
(where $\dot{G}$ is the $\mathbb{P}$-name for its generic) such that $\pi \restriction H^{V[\dot{G}]}_\theta \in M$, $\text{crit}(\pi) = \omega_2^{V[\dot{G}]}$, $H^{V[\dot{G}]}_\theta$ is in the wellfounded part of $M$, and there is an $H \in M$ that is $(V[\dot{G}],\mathbb{Q})$-generic such that $q \in H$ and for every $i < \mu$, $M \models$ $(\dot{S}_i)_{H}$ is a stationary subset of $\omega_1$."

So let $p$ be an arbitrary condition in $\mathbb{P}$, $\dot{\mathbb{Q}}$ a $\mathbb{P}$-name for a poset in $\dot{\Delta}$, $\dot{q}$ a $\mathbb{P}$-name for a condition in $\dot{\mathbb{Q}}$, and $\dot{\vec{S}} = \langle \dot{S}_i \ : \ i < \mu \rangle$ a $\mu$-sequence of $\mathbb{P}*\dot{\mathbb{Q}}$-names for stationary subsets of $\omega_1$.  Let $\dot{\mathbb{R}}$ be as in the hypotheses of Theorem \ref{thm_PreservationFA}.  Fix a regular $\theta$ such that all objects mentioned so far are in $H_\theta$.  Note that by the assumptions about $\dot{\mathbb{R}}$, each $\dot{S}_i$ remains stationary in $V^{\mathbb{P}*\dot{\mathbb{Q}}*\dot{\mathbb{R}}}$; so $\dot{\vec{S}}$ can be regarded as a sequence of $\mathbb{P}*\dot{\mathbb{Q}}*\dot{\mathbb{R}}$-names of stationary subsets of $\omega_1$.  Since  $\mathbb{P}*\dot{\mathbb{Q}}*\dot{\mathbb{R}}$ is an element of $\Gamma$ by assumption, Theorem \ref{thm_VariantWoodinTheorem} ensures that there is, in some generic extension $W$ of $V$, a generic elementary embedding $j: V \to N$ such that 
\begin{enumerate}
 \item $H^V_{(2^\theta)^+}$ is in the wellfounded part of $N$, and has size $\omega_1$ in $N$;
 \item $j \restriction H^V_{(2^\theta)^+} \in N$; 
 \item $\text{crit}(j) = \omega_2^V$;
 \item\label{item_YetAnotherStatCorrect} There is a $G*H*K \in N$ that is $V$-generic for $\mathbb{P}*\dot{\mathbb{Q}}*\dot{\mathbb{R}}$, and such that $(p,\dot{q}, \dot{1}) \in G*H*K$ and $N \models$ ``$(\dot{S}_i)_{G*H*K}=(\dot{S}_i)_{G*H}$ is a stationary subset of $\omega_1$, for all $i < \mu$."
\end{enumerate}

Note that $j[G] \in N$ because $G \in N$ and $j \restriction H^V_{(2^\theta)^+} \in N$.  So $j:V \to N$ and $G*H*K$ satisfy all the hypothesis in the ``if-then" clause \ref{item_ClauseWhenever} of the theorem we are currently proving.  So by that clause, there is some $p' \in j(\mathbb{P})$ that $N$ believes is a lower bound of $j[G]$.  Now $p'$ may be in the illfounded part of $N$, but this will not matter.  Recall that $W$ is the generic extension of $V$ where the embedding $j: V \to N$ resides.  Work in $W$ for the moment.  Then even if $N$ is illfounded, the fact that $j(\mathbb{P})$ is a partial order is upward absolute to $W$, and $W$-generics are also $N$-generics.  More precisely, in $W$ consider the $\in_N$-extension $\mathcal{P}'$ of $j(\mathbb{P})$---i.e.\ $\mathcal{P}':= \{ x \in N \  : \ x \in^N j(\mathbb{P}) \}$---and order $\mathcal{P}'$ by $x \le y$ iff $N \models x \le_{j(\mathbb{P})} y$; it is routine to check that $\mathcal{P}'$ is a partial order in $W$.  Choose a $(W,\mathcal{P}')$-generic $G'$ such that $p' \in G'$.  Then again it is routine to see that $G'$ is also $\big( W, j(\mathbb{P}) \big)$-generic, in the sense that for every $D \in N$ such that $N \models$ ``$D$ is dense in $j(\mathbb{P})$" there is some $q \in G'$ such that $q \in^N D$.

Now consider the generic extension $N[G']$.  A word is in order about what is meant by $N[G']$, since the standard recursive definition of name-evaluation by a generic does not make sense when the ground model is illfounded.  One way to make sense of $N[G']$ is via quotients of Boolean-valued models.\footnote{Alternatively, one could use the fact that $j \restriction H^V_\theta \in N$ to work entirely \emph{within} $N$---more precisely, with the forcing relation of $N$---for the remainder of the proof. This alternative would involve extending $j \restriction H^V_\theta: H^V_\theta \to N$ generically, rather than extending the entire map $j:V \to N$ as is done below.  But the result would be the same, namely, in this alternative context one would use generic liftings of $j \restriction H^V_\theta$ to prove that $H^V_\theta[G] \models \text{FA}(\mathbb{Q})$, which would imply that $V[G] \models \text{FA}(\mathbb{Q})$.}  In the current context, let $\mathbb{B}:= \text{ro}^N\big( j(\mathbb{P}) \big)$ and consider the quotient $N^{\mathbb{B}}/ G'$ of the $\mathbb{B}$-valued model $N^{\mathbb{B}}$ (see Hamkins-Seabold~\cite{HamkinsSeabold}).  Then since $G'$ is generic over $N$, Lemma 13 and Theorem 16 of \cite{HamkinsSeabold} ensure that there is a definable isomorphic copy of $N$ inside $N^{\mathbb{B}}/ G'$---which we identify with $N$---such that $N^{\mathbb{B}}/G' \models$ ``I am the forcing extension of $N$ by $G'$."  Viewed in this way, $N[G']$ is simply $N^{\mathbb{B}}/G'$, and if $\tau$ is a $j(\mathbb{P})$-name in $N$, then the name evaluation $\tau_{G'}$ makes sense if computed from the point of view of $N^{\mathbb{B}}/ G'$ (which believes that $N$ is wellfounded).

Since $p' \in G'$ and $p'$ was stronger than every element of $j [G]$, elementarity of $j:V \to N$ ensures that $j[G] \subseteq G'$.  Then the standard argument shows that the map 
\[
\hat{j}:V[G] \to N[G']
\]
 defined by $\sigma_G \mapsto j(\sigma)_{G'}$ is well-defined and elementary (here $j(\sigma)_{G'}$ is computed in $N^{\mathbb{B}}/ G' =N[G']$, which makes sense as described above).  Note that $\hat{j}$ is definable in the extension $W[G']$, which (since $G \in W$ and $W$ was a generic extension of $V$) is a generic extension of $V[G]$.

Next we observe that $\mathbb{P}$ must have been $< \omega_2$-distributive; this will guarantee that $\text{crit}(\hat{j}) = \text{crit}(j) = \omega_2^V = \omega_2^{V[G]}$.  If $\mathbb{P}$ were not $<\omega_2$-distributive, then there would be some name $\dot{\vec{\alpha}} = \langle \dot{\alpha}_i \ : \ i < \omega_1^V \rangle \in H^V_\theta$ for a new $\omega_1^V$-length sequence of ordinals, and hence $j(\dot{\vec{\alpha}}) = \langle j(\dot{\alpha}_i) \ : \ i < \omega_1^V \rangle$ would be a $j(\mathbb{P})$-name for a new $j(\omega_1^V) = \omega_1^V$-length sequence of ordinals.  Now $N$ sees that the condition $p'$ extends the $\big(j[H^V_\theta], j(\mathbb{P}) \big)$-generic filter $j[G]$, and hence that $p'$ is a \emph{total master condition} for the model $j[H^V_\theta]$ (i.e.\ for every $D \in j[H^V_\theta]$ such that $D$ is dense in $j(\mathbb{P})$, there is a condition weaker than $p'$ in $D \cap j[H^V_\theta]$).  It follows that $p'$ already decides all values from the sequence $\langle j(\dot{\alpha}_i) \ : \ i < \omega_1^V \rangle$.  Since $p' \in G'$, $N[G']$ sees that $\hat{j}(\dot{\vec{\alpha}}_G) = \big( j(\dot{\vec{\alpha}})  \big)_{G'}$ is already an element of the ground model $N$.  By elementarity of $\hat{j}$, $V[G]$ sees that $\dot{\vec{\alpha}}_G$ is already an element of the ground model $V$, which is a contradiction.

Let $\mathbb{Q}:=\dot{\mathbb{Q}}_G$.  Since $j \restriction  H^V_\theta \in N$ and $G*H \in N$, it follows that $H \in N[G']$ and
\[
\hat{j} \restriction ( H^V_\theta[G]) = \langle \sigma_G \mapsto j(\sigma)_{G'} \ : \ \sigma \in H^V_\theta \cap N^{j(\mathbb{P})}  \rangle
\]
 is an element of $N[G']$. Furthermore, $q:=\dot{q}_G$ is an element of $H$, by \eqref{item_YetAnotherStatCorrect}.

Finally, we need to verify that from the point of view of $N[G']$, each $(\dot{S}_i)_{G*H}$ is stationary.  Indeed, each $(\dot{S}_i)_{G*H}$ is stationary from the point of view of $N$, by \eqref{item_YetAnotherStatCorrect}.  Furthermore, since $\mathbb{P}$ is a regular suborder of $\mathbb{P}*\dot{\mathbb{Q}}*\dot{\mathbb{R}}$ and the latter is in $\Gamma$, Observation \ref{obs_FMS} ensures that $\mathbb{P}$ is $\omega_1$-SSP from the point of view of $V$.  By elementarity of $j$, $j(\mathbb{P})$ is $\omega_1$-SSP from the point of view of $N$.  Hence, each $(\dot{S}_i)_{G*H}$ remains stationary in $N[G']$.
\end{proof}

Theorem \ref{thm_PreservationFA} gives an alternative proof of several preservation results in the literature.  We highlight a few below.

\subsection{Examples from the literature}\label{sec_ExamplesFApres}

\subsubsection{The Larson example and generalizations}\label{sec_LarsonExample}

Larson's~\cite{MR1782117} theorem that MM is preserved by $<\omega_2$-directed closed forcing can be generalized to show the same is true for virtually \emph{any} of the standard forcing axioms ($\text{MA}_{\omega_1}$, PFA, MM, fragments of MM, etc.) and their ``$+\mu$" versions:
\begin{theorem}
Suppose $\Gamma$ is closed under restrictions and under 2-step iterations, and let $\mu$ be any cardinal $\le \omega_1$.  Then $\text{FA}^{+\mu}(\Gamma)$ is preserved by $<\omega_2$-directed closed forcing.  
\end{theorem}
\begin{proof}
Suppose $V \models \text{FA}^{+\mu}(\Gamma)$, $\mathbb{P}$ is $<\omega_2$-directed closed, $\dot{\mathbb{Q}}$ is a $\mathbb{P}$-name for a poset in the class $\dot{\Gamma}$, and $\dot{\vec{S}} = \langle \dot{S}_i \ : \ i < \mu \rangle$ is a $\mu$-sequence of $\mathbb{P}*\dot{\mathbb{Q}}$-names for stationary subsets of $\omega_1$.  Since $\Gamma$ is closed under 2-step iterations, $\mathbb{P}*\dot{\mathbb{Q}} \in \Gamma$.  Let $\dot{\mathbb{R}}$ be the $\mathbb{P}*\dot{\mathbb{Q}}$-name for the trivial forcing.  Then clearly $\mathbb{P}*\dot{\mathbb{Q}} (*\dot{\mathbb{R}}) \in \Gamma$ also.  Note that $\dot{\mathbb{R}}$ trivially preserves the stationarity of each $\dot{S}_i$.  Thus,  clauses \ref{item_3stepinGamma} and \ref{item_R_PreserveStat} of the hypotheses of Theorem \ref{thm_PreservationFA} are satisfied.

We now verify that clause \ref{item_ClauseWhenever} of the assumptions of Theorem \ref{thm_PreservationFA} holds.  Suppose $j: V \to N$ and $G*H (*K)$ are as in the hypotheses of clause \ref{item_ClauseWhenever} of Theorem \ref{thm_PreservationFA}.  Since $|H^V_\theta|^N = \omega_1$, $\mathbb{P} \in H^V_\theta$, $G$ is a filter, and $j \restriction H^V_\theta \in N$, then $j[G] \in N$ and $N$ believes that $j[G]$ is a directed subset of $j(\mathbb{P})$ of size $<\aleph_2$.  Moreover, by elementarity of $j$, $N$ believes that $j(\mathbb{P})$ is $<\aleph_2$-directed closed.  Hence $N$ believes $j[G]$ has a lower bound.  So clause \ref{item_ClauseWhenever} of Theorem \ref{thm_PreservationFA} is also satisfied.   So, Theorem \ref{thm_PreservationFA} ensures that $V^{\mathbb{P}} \models \text{FA}^{+\mu}(\dot{\Gamma})$.
\end{proof}

\subsubsection{The Beaudoin-Magidor example}\label{sec_BedMagExample}

The classic Beaudoin-Magidor theorem that PFA is consistent with a nonreflecting stationary subset of $S^2_0$ (see \cite{MR1099782}) can be re-proved as follows (with $\Gamma=$ the class of proper posets):  assume PFA in $V$, and let $\mathbb{P}$ be the forcing to add a nonreflecting stationary subset $\dot{S}$ of $S^2_0$ with initial segments.  Let $\dot{\mathbb{Q}}$ be a $\mathbb{P}$-name for a proper poset.  Let $\dot{\mathbb{R}}$ be the $\mathbb{P}*\dot{\mathbb{Q}}$-name for the poset to kill the stationarity of $\dot{S}$ using countable conditions.  Then $\mathbb{P}*\dot{\mathbb{Q}}*\dot{\mathbb{R}}$ is proper (see \cite{MR1099782}).  Moreover, if $j: V \to N$ and $S*H*K \in N$ are as in the hypotheses of clause \ref{item_ClauseWhenever} of Theorem \ref{thm_PreservationFA}, the presence of the $\dot{\mathbb{R}}$-generic club $K$ ensures that $S = j[ S]$ is nonstationary in $\omega_2^V$, and hence is a condition in $j(\mathbb{P})$ (and clearly a lower bound for $j[S]$).  So clause \ref{item_ClauseWhenever} of Theorem \ref{thm_PreservationFA} is satisfied, and so PFA holds in $V[S]$. 

The poset $\mathbb{P}$ in this example is actually a member of the class of posets considered in the next example.

\subsubsection{The Yoshinobu example}\label{sec_YoshinobuExample} 
Yoshinobu's~\cite{MR3037549} theorem about preservation of PFA by $\omega_1 + 1$ operationally closed forcings can also be viewed as a consequence of Theorem \ref{thm_PreservationFA}.  Suppose PFA holds and $\mathbb{P}$ is $\omega_1 + 1$ operationally closed; roughly, this means that in the game $\mathfrak{G}_{\mathbb{P}}$ of length $\omega_1 + 1$ where the players create a descending chain of conditions with Player II playing at limit stages (and losing if she cannot play at some limit stage $\alpha \le \omega_1$), Player II has a winning strategy $\sigma$ that only depends on the ``current position" of the game (i.e.\ on the boolean infimum of the conditions played so far and the current ordinal stage, but \emph{not} on the history of the game so far).  Let $\dot{\mathbb{Q}}$ be a $\mathbb{P}$-name for a proper poset.  Using properness of $\dot{\mathbb{Q}}$ and the fact that $\sigma$ only depends on the current position, Yoshinobu designs a $\mathbb{P}*\dot{\mathbb{Q}}$-name $\dot{\mathbb{R}}$ such that the 3-step iteration is proper, and if $G*H*K$ is generic over $V$ for $\mathbb{P}*\dot{\mathbb{Q}}*\dot{\mathbb{R}}$, then in $V[G*H*K]$ there exists a play $\mathcal{P}$ of length $\omega_1$ such that all proper initial segments of $\mathcal{P}$ are in $V$, player II used $\sigma$ at every countable stage of $\mathcal{P}$, and the conditions played in $\mathcal{P}$ generate $G$.  Then, if $G*H*K \in N$ and $j: V \to N$ are as in clause \ref{item_ClauseWhenever} of Theorem \ref{thm_PreservationFA}---so in particular $j \restriction H^V_\theta$, $j[\mathcal{P}]$, and $j[G]$ are elements of $N$---then $N$ believes that $j[\mathcal{P}]$ is a $j(\omega_1) = \omega_1$-length play of the game $j(\mathfrak{G}_{\mathbb{P}})$ where player II used $j(\sigma)$ at every countable stage, and that $j[G]$ is generated by the conditions played in $j[\mathcal{P}]$.  Hence, since $j(\sigma)$ was used along the way, the conditions played in $j[\mathcal{P}]$ have a lower bound, which (since $j[G]$ is generated by $j[\mathcal{P}]$) is also a lower bound for $j[G]$.  Then $V^{\mathbb{P}} \models \text{PFA}$ by Theorem \ref{thm_PreservationFA}.

\section{Separation of appproachability properties}\label{sec_SeparationApproach}

The notions of \textbf{disjoint club} and \textbf{disjoint stationary} sequences on $\omega_2$ were introduced in Friedman-Krueger~\cite{MR2276627} and Krueger~\cite{MR2502487}, respectively.  We say that $\omega_2$ carries a disjoint club (\emph{resp.\ disjoint stationary}) sequence iff there is a stationary $S \subset S^2_1$ and a sequence $\langle x_\gamma \ : \ \gamma \in S  \rangle$ such that every $x_\gamma$ is a club (\emph{resp. stationary}) subset of $[\gamma]^\omega$, and $x_\gamma \cap x_{\gamma'} = \emptyset$ whenever $\gamma \ne \gamma'$ are both in $S$.  Krueger proved:

\begin{theorem}[Theorems 6.3 and 6.5 of Krueger~\cite{MR2502487}]\label{thm_Krueger_DisjointSequences}
Assume $2^{\omega_1} = \omega_2$.
\begin{itemize}
 \item The existence of a disjoint club sequence on $\omega_2$ is equivalent to the assertion that $\text{IU} \ne \text{IS}$ in $\wp_{\omega_2}(H_{\omega_2})$; i.e.\ there are stationarily many $W \in \wp^*_{\omega_2}(H_{\omega_2})$ that are internally unbounded, but not internally stationary.
 \item The existence of a disjoint stationary sequence on $\omega_2$ is equivalent to the assertion that $\text{IU} \ne \text{IC}$ in $\wp_{\omega_2}(H_{\omega_2})$; i.e.\ there are stationarily many $W \in \wp^*_{\omega_2}(H_{\omega_2})$ that are internally unbounded, but not internally club.
\end{itemize}
\end{theorem}

\subsection{Proof of Theorem \ref{thm_Cox_PFAplus1necessary}}\label{sec_ProofThmCox_PFAplus1necessary}

In this section we show that PFA does not imply the existence of a disjoint stationary sequence on $\omega_2$ (and thereby show that Krueger's Theorem \ref{thm_Krueger_PFAplus1} is sharp).  By Theorem \ref{thm_Krueger_DisjointSequences}, it suffices to find a model of PFA that also satisfies
\[
\text{IC} \cap \wp_{\omega_2}(H_{\omega_2}) =^* \text{IU} \cap \wp_{\omega_2}(H_{\omega_2}).
\]

Assume $V$ is a model of PFA; then $2^{\omega_1} = \omega_2$ so we can fix a bijection $\Phi: \omega_2 \to H_{\omega_2}$.  Let $\text{IC}^*$ be the set of $\gamma \in S^2_1$ such that $\Phi[\gamma] \cap \omega_2 = \gamma$ and $\Phi[\gamma]$ is an IC model.  It is routine to check that all but nonstationarily many $W \in \text{IC} \cap \wp_{\omega_2}(H_{\omega_2})$ are of the form $\Phi[W \cap \omega_2]$ where $W \cap \omega_2 \in \text{IC}^*$.  

Let $\mathbb{C}(\text{IC}^*)$ be the poset of closed, bounded $c \subset \omega_2$ such that $c \cap S^2_1 \subset \text{IC}^*$, ordered by end-extension.  By an argument similar to Proposition 4.4 of Krueger~\cite{MR2674000}, $\mathbb{C}(\text{IC}^*)$ is $<\omega_2$ distributive.\footnote{Briefly: if $W \prec (H_{\omega_3},\in,\Phi)$ and $W \in \text{IA}$, then $W \cap H_{\omega_2} = \Phi[W \cap \omega_2]$, $W \cap \omega_2 \in \text{IC}^*$ (since $\text{IA} \subseteq \text{IC}$), and any condition in $W \cap \mathbb{C}(\text{IC}^*)$ can be extended to a $W$-generic tower with supremum $W \cap \omega_2$ (this argument uses internal approachability of $W$).  Then since $W \cap \omega_2 \in \text{IC}^*$, that generic tower has a lower bound, obtained by placing $W \cap \omega_2$ at the top of the tower.  In summary, every condition in $W$ can be extended to a condition whose upward closure is a $(W,\mathbb{C}(\text{IC}^*)$-generic filter.  Since the set of such $W$ is stationary, standard arguments then imply that $\mathbb{C}(\text{IC}^*)$ is $<\omega_2$ distributive.  }  In particular $H_{\omega_2}$ is unchanged.  Then the club added by $\mathbb{C}(\text{IC}^*)$ witnesses that, in $V^{\mathbb{C}(\text{IC}^*)}$, almost every $\gamma \in S^2_1$ has the property that $\Phi[\gamma] \in \text{IC}$, and hence that almost every element of $\text{IU} \cap \wp_{\omega_2}(H_{\omega_2})$ is internally club.  

It remains to show that $V^{\mathbb{C}(\text{IC}^*)} \ \models \text{PFA}$.  Let $\dot{\mathbb{Q}}$ be a $\mathbb{C}(\text{IC}^*)$-name for a proper poset.  The poset $\mathbb{C}(\text{IC}^*)$ is $\sigma$-closed, and hence $\mathbb{C}(\text{IC}^*) * \dot{\mathbb{Q}}$ is proper.  In particular, $V \cap  [\omega_2^V]^\omega$ remains stationary.  Let $\dot{\mathbb{R}}$ be the $\mathbb{C}(\text{IC}^*) * \dot{\mathbb{Q}}$-name for the poset from Definition \ref{def_DecoratedPoset} that shoots a continuous $\omega_1$-chain through
$V \cap  [\omega_2^V]^\omega$.\footnote{For this application we could just as well have used a version of that poset using countable conditions, but we choose to stick with Definition \ref{def_DecoratedPoset} since it works just as well.}  By Lemma \ref{lem_ClubThruVisProper}, the poset $\mathbb{C}(\text{IC}^*) * \dot{\mathbb{Q}} * \dot{\mathbb{R}}$ is proper.  

Now suppose $j: V \to N$ is a generic elementary embedding with the properties listed in clause \ref{item_ClauseWhenever} of Theorem \ref{thm_PreservationFA} (with $\mathbb{C}(\text{IC}^*)$ playing the role of the $\mathbb{P}$ from that theorem).  More precisely, assume $\text{crit}(j) = \omega_2^V$, $j \restriction H^V_\theta \in N$, $H^V_\theta$ is in the wellfounded part of $N$, and there is a $C*H*K \in N$ that is $V$-generic for $\mathbb{C}(\text{IC}^*) * \dot{\mathbb{Q}} * \dot{\mathbb{R}}$.  The presence of $K$ ensures that, in $H^V_\theta[C*H*K]$, $H^V_{\omega_2} \in \text{IC}$.  By Lemma \ref{lem_InternalPartAbsolute} this is upward absolute to $N$.  Furthermore, $j(\Phi)[\omega_2^V] = \text{range}(\Phi) = H^V_{\omega_2}$.  So $N \models \omega_2^V \in j(\text{IC}^*)$, and hence
\begin{equation}
N \models C \cup \{ \omega_2^V \} \in j(\mathbb{C}(\text{IC}^*)).
\end{equation}
and is clearly stronger than $j[C]=C$.  By Theorem \ref{thm_PreservationFA}, $V[C] \models \text{PFA}$, completing the proof of Theorem \ref{thm_Cox_PFAplus1necessary}.

\subsection{Proof of Theorem \ref{thm_Cox_MMnecessary}}

In this section we prove that $\text{PFA}^{+\omega_1}$ does not imply the existence of a disjoint club sequence on $\omega_2$, thereby showing that Krueger's Theorem \ref{thm_Krueger_MM} is sharp.  By Theorem \ref{thm_Krueger_DisjointSequences}, it suffices to construct a model of $\text{PFA}^{+\omega_1}$ such that $\text{IS}=^* \text{IU}$ in $\wp_{\omega_2}(H_{\omega_2})$.  

Assume $V$ is a model of $\text{PFA}^{+\omega_1}$.  Fix a bijection $\Phi: \omega_2 \to H_{\omega_2}$ and let $\text{IS}^*$ be the set of $\gamma \in S^2_1$ such that $\Phi[\gamma] \in \text{IS}$.  Then almost every $W \in \text{IS} \cap \wp_{\omega_2}(H_{\omega_2})$ is of the form $\Phi[W \cap \omega_2]$ where $W \cap \omega_2 \in \text{IS}^*$.  

Let $\mathbb{C}(\text{IS}^*)$ be the poset of closed, bounded $c \subset \omega_2$ such that $c \cap S^2_1 \subset \text{IS}^*$, ordered by end-extension.  As in Section \ref{sec_ProofThmCox_PFAplus1necessary}, this poset is $<\omega_2$ distributive, and in particular $H_{\omega_2}$ and $\text{IS} \cap H_{\omega_2}$ are computed the same in $V$ and $V^{\mathbb{C}(\text{IS}^*)}$.  It is also routine to check that
\begin{equation}
V^{\mathbb{C}(\text{IS}^*)} \ \models \ \text{IS}=^* \text{IU} \text{ in } \wp_{\omega_2}(H_{\omega_2}).
\end{equation}

It remains to show that $V^{\mathbb{C}(\text{IS}^*)}$ is a model of $\text{PFA}^{+\omega_1}$.  Let $\dot{\mathbb{Q}}$ be a $\mathbb{C}(\text{IS}^*)$-name for a proper poset, let $\theta$ be a large regular cardinal with $\mathbb{C}(\text{IS}^*)*\dot{\mathbb{Q}} \in H_\theta$.  

By Theorem \ref{thm_PreservationFA},\footnote{Viewing $\mathbb{C}(\text{IS}^*)$ as the $\mathbb{P}$ from that theorem, and taking $\dot{\mathbb{R}}$ to be the $\mathbb{C}(\text{IS}^*)*\dot{\mathbb{Q}}$-name for the trivial poset (which is trivially forced by $\mathbb{C}(\text{IS}^*)*\dot{\mathbb{Q}}$ to be $\omega_1$-SSP).} in order to show that $\text{PFA}^{+\omega_1}$ holds in $V^{\mathbb{C}(\text{IS}^*)}$ it suffices to prove that \textbf{if} $j: V \to N$ is a generic elementary embedding with critical point $\omega_2^V$ such that $j \restriction H_\theta^V \in N$, $H^V_\theta$ is in the wellfounded part of $N$ and has size $\omega_1$ in $N$, and $N$ has an $\omega_1$-stationary correct $V$-generic filter $C*H$ for $\mathbb{C}(\text{IS}^*)*\dot{\mathbb{Q}}$, \textbf{then} $N$ believes that $j[C] = C$ has a lower bound in $j\big(\mathbb{C}(\text{IS}^*) \big)$.  Now since $j(\Phi)[\omega_2^V] = \text{range}(\Phi) = H^V_{\omega_2}$, it suffices to show that $H^V_{\omega_2} \in \text{IS}^N$ (since then $C \cup \{ \omega_2^V \}$ will be a condition in $j\big(\mathbb{C}(\text{IS}^*) \big)$ below $C = j[C]$).  We can without loss of generality assume that $\dot{\mathbb{Q}}$ collapses $\omega_2^V$.  Since $\mathbb{C}(\text{IS}^*)*\dot{\mathbb{Q}}$ is proper, it forces that $V \cap [\omega_2^V]^\omega$ is stationary; and since it collapses $\omega_2$, it follows that $V[C*H] \models H^V_{\omega_2} \in \text{IS}$.  Let $\dot{T}$ be the $\mathbb{C}(\text{IS}^*)*\dot{\mathbb{Q}}$-name for the internal part of $H^V_{\omega_2}$; then $\dot{T}$ is forced to be a stationary subset of $\omega_1$.  Since $C*H$ is an $\omega_1$-stationary correct filter, $\dot{T}_{C*H}$ is stationary in $N$.  Hence $H^V_{\omega_2}$ is internally stationary from the point of view of $N$.

\subsection{Proof of Theorem \ref{thm_Cox_DontNeedPlus2}}\label{sec_MM_Viale}

We need to prove that the global separation of IC from IS:
\begin{enumerate*}
 \item  follows from MM; 
 \item follows from $\text{PFA}^+$; and
 \item does not follow from PFA.
\end{enumerate*}
That PFA does not imply separation of IC from IS already follows from our Theorem \ref{thm_Cox_PFAplus1necessary}; recall in Section \ref{sec_ProofThmCox_PFAplus1necessary} we obtained a model of PFA where $\text{IC}=^* \text{IU}$ in $\wp_{\omega_2}(H_{\omega_2})$.  So it remains to show either MM or $\text{PFA}^+$ implies global separation of IS from IC.

We start with the MM proof.  We prove a little more, namely:

\begin{theorem}\label{thm_FixViale}
MM implies that $\text{G}_{\omega_1} \cap \text{IS} \setminus \text{IC}$ is stationary for every regular $\theta \ge \omega_2$. 
\end{theorem}

Theorem \ref{thm_FixViale} appeared as Theorem 4.4 part (3) of Viale~\cite{Viale_GuessingModel}; however the proof given there is incorrect (the proof given there appears to be implicitly using the stronger assumption $\text{MM}^{+2}$, analogous to Krueger's proof in \cite{MR2674000}, and works fine under that stronger assumption.)  We briefly summarize the argument from \cite{Viale_GuessingModel}, and where the error occurs.  A certain poset $\mathbb{P}_2 * \dot{\mathbb{Q}}_{\mathbb{P}_2}$ is defined that forces $H^V_\theta$ to be a guessing model in $\text{IS} \setminus \text{IC}$.  The third bullet at the very end of that proof (near the end of Section 4) claims that if $W \prec H_\theta$ is such that $|W|=\omega_1 \subset W \prec H_{(2^\theta)^+}$ and there exists a $(W, \mathbb{P}_2 * \dot{\mathbb{Q}}_{\mathbb{P}_2})$-generic filter, then $W$ is guessing, and $W \in \text{IS} \setminus \text{IC}$.  The guessing part is correct, but such a $W$ is not necessarily in $\text{IS} \setminus \text{IC}$.  To see why, let $\dot{\mathbb{C}}=\dot{\mathbb{C}}^{\text{fin}}_{\text{dec}}(V \cap [\theta]^\omega)$ be the $\mathbb{P}_2 * \dot{\mathbb{Q}}_{\mathbb{P}_2}$-name for the poset to shoot a continuous $\omega_1$-chain through $V \cap [\theta]^\omega$ as in Definition \ref{def_DecoratedPoset}.  Then $\mathbb{P}_2 * \dot{\mathbb{Q}}_{\mathbb{P}_2}*\dot{\mathbb{C}}$ is proper, and forces $H^V_\theta \in \text{IC}$.\footnote{And is still $\omega_1$-guessing, since $\mathbb{P}_2 * \dot{\mathbb{Q}}_{\mathbb{P}_2}$ forces $H^V_\theta$ to be indestructibly guessing.}  Hence under MM (or just PFA) there are stationarily many $W$ as above such that there exists a $(W, \mathbb{P}_2 * \dot{\mathbb{Q}}_{\mathbb{P}_2}*\dot{\mathbb{C}})$-generic filter; say $g*h*c$.  So in particular (by projecting to the first 2 coordinates) there exists a $W$-generic for $\mathbb{P}_2 * \dot{\mathbb{Q}}_{\mathbb{P}_2}$.  But the 3rd coordinate $c$ ensures that $W \cap H_\theta$ is an element of $\text{IC}$, and hence \emph{not} an element of $\text{IS} \setminus \text{IC}$.   

(One could give a different counterexample under MM, by instead letting $\dot{\mathbb{C}}$ name the poset to shoot a continuous $\omega_1$ chain through $[\theta]^\omega \setminus V$; in that context one would be able to find $W$ for which a $\mathbb{P}_2 * \dot{\mathbb{Q}}_{\mathbb{P}_2}$-generic exists, yet $W \notin \text{IS}$).

We now give a proof of Theorem \ref{thm_FixViale}.  Most of the work was already done in Section \ref{sec_ControlInternal} above.  
\begin{proof}
(of Theorem \ref{thm_FixViale}):  Assume MM, and fix a $T \subset \omega_1$ that is stationary and costationary.  Let $\theta \ge \omega_2$ be regular, and let $\mathbb{Q}^{\omega_1 \text{-SSP}}_{T,\theta}$ be the poset given by Theorem \ref{thm_TwoPosets}.  By Woodin's Lemma \ref{lem_WoodinKeyLemma}, there are stationarily many $W \prec H_{(2^\theta)^+}$ for which there exists a $(W, \mathbb{Q}^{\omega_1 \text{-SSP}}_{T,\theta})$-generic filter.  By Corollary \ref{cor_ControlInternal}, $W \cap H_\theta$ is an (indestructible) guessing model, and the internal part of $W$ is exactly $T$.  In particular, $W \in \text{IS} \setminus \text{IC}$, since $T$ was stationary and costationary.
\end{proof}

Finally, we show that the conclusion of Krueger's Theorem \ref{thm_Krueger_PFAplus2} follows from $\text{PFA}^+$.  Again we prove something a little stronger, namely:

\begin{theorem}\label{thm_FromPFAplus1}
$\text{PFA}^{+1}$ implies that $\text{G}_{\omega_1} \cap \text{IS} \setminus \text{IC}$ is stationary for every regular $\theta \ge \omega_2$. 
\end{theorem}

\begin{proof}
Again, most of the work was done in Section \ref{sec_ControlInternal} above.  Assume $\text{PFA}^+$, fix a regular $\theta \ge \omega_2$, and fix any stationary, costationary $T \subset \omega_1$.  Let $\mathbb{Q}^{\text{proper}}_{T, \theta}$ be the poset given by Theorem \ref{thm_TwoPosets}, and let $\dot{S}_{\text{external}}$ be the $\mathbb{Q}^{\text{proper}}_{T, \theta}$-name for the external part of $H^V_\theta$.  By Woodin's Lemma \ref{lem_WoodinKeyLemma}, there are stationarily many $W \prec H_{(2^\theta)^+}$ such that $|W|=\omega_1 \subset W$ and there exists a $(W,\mathbb{Q}^{\text{proper}}_{T, \theta})$-generic filter $g$ that interprets $\dot{S}_{\text{external}}$ as a stationary subset of $\omega_1$.   By Corollary \ref{cor_ControlInternal}, $W \cap H_\theta$ is indestructibly $\omega_1$-guessing, its internal part contains the stationary set $T$ (hence $W \cap H_\theta \in \text{IS}$) and its external part is a stationary subset of $T^c$ (hence $W \cap H_\theta \notin \text{IC}$).
\end{proof}

\section{Separation of stationary set reflection}\label{sec_SeparateStatReflect}

In this section we prove the various separations of stationary reflection from the introduction.  The proofs of Theorems \ref{thm_PFAplusNotGICimpliesDRPGIS}, \ref{thm_PFAplusNotGIAimpliesDRPGIC}, and \ref{thm_MMplusnotISimpliesORD_DRP} are similar, so we group those together.  Most of the work for these theorems was done already in Section \ref{sec_ControlInternal}.  The proofs of Theorems \ref{thm_Cox_PreservePFAnegIC} and \ref{thm_Cox_PreservePFAnegIA} are also similar, so those are grouped together; the main tool for those is Theorem \ref{thm_PreservationFA}.  In the final subsection we prove Theorem \ref{thm_MM_notIS_notImplyRPIU}.

\subsection{Forcing axiom implications}

Here we prove Theorems \ref{thm_PFAplusNotGICimpliesDRPGIS}, \ref{thm_PFAplusNotGIAimpliesDRPGIC}, and \ref{thm_MMplusnotISimpliesORD_DRP}.

\subsubsection{Proof of Theorem \ref{thm_PFAplusNotGICimpliesDRPGIS}}\label{subsub_PFAnotGIC}

Assume $\text{PFA}^{+\omega_1}_{H^V_{\omega_2} \notin \text{IC}}$, and fix a regular $\theta \ge \omega_2$.  Fix a costationary subset $T$ of $\omega_1$ ($T  = \emptyset$ will work).  Let $\mathbb{Q}^{\text{proper}}_{\theta,T}$ be the poset from Theorem \ref{thm_TwoPosets}; that theorem (and costationarity of $T$) ensures that $\mathbb{Q}^{\text{proper}}_{\theta,T}$ is proper and forces $H^V_{\omega_2} \in \text{IS} \setminus \text{IC}$, so in particular is in the class of posets to which $\text{PFA}^{+\omega_1}_{H^V_{\omega_2} \notin \text{IC}}$ applies.  By Woodin's Lemma \ref{lem_WoodinKeyLemma}, there are stationarily many $W \prec (H_{(2^\theta)^+},\in,\dots)$ such that $|W|=\omega_1 \subset W$ and there exists a $(W,\mathbb{Q}^{\text{proper}}_{\theta,T})$-generic filter $g$ that is correct with respect to stationary subsets of $\omega_1$ (that have names in $W$).  By Corollary \ref{cor_ControlInternal}, $W \cap H_\theta$ is a guessing set, and is in $\text{IS} \setminus \text{IC}$.

It remains to show that $W \cap H_\theta$ is diagonally, \emph{internally} reflecting, so let $R \in W$ be a stationary subset of $[H_\theta]^\omega$.  We need to prove that $R \cap W \cap [W \cap H_\theta]^\omega$ is stationary in $[W \cap H_\theta]^\omega$.  Let $\langle \dot{Q}_i \ : \ i < \omega_1 \rangle$ be the name for the generic filtration of $H^V_\theta$, and $\dot{S}_R$ be the name for
\[
\{ i  < \omega_1 \ : \ \dot{Q}_i \in \check{R}  \}.
\]
Since $\mathbb{Q}^{\text{proper}}_{\theta,T}$ is proper, then $R$ remains stationary, and it follows that $\dot{S}_R$ names a stationary subset of $\omega_1$.  Also notice that since $R \in V$, then in particular $R \subset V$ and so
\begin{equation}\label{eq_ForceQi_in_V}
\Vdash \forall i \in \dot{S}_R \ \ \dot{Q}_i \text{ is an element of the ground model}.
\end{equation}

Now the name $\dot{S}_R$ can be taken to be an element of $W$, so $S:=(\dot{S}_R)_g$ is really stationary in $V$.  Now $\langle (\dot{Q}_i)_g \ : \ i < \omega_1 \rangle$ is a filtration of $W \cap H_\theta$.  Then for every $i \in S$,  $(\dot{Q}_i)_g$ is an element of $R$, and also (by \eqref{eq_ForceQi_in_V}) an element of $W$.  This completes the proof of Theorem \ref{thm_PFAplusNotGICimpliesDRPGIS}.

\subsubsection{Proof of Theorem \ref{thm_PFAplusNotGIAimpliesDRPGIC}}

This is very similar to the proof of Theorem \ref{thm_PFAplusNotGICimpliesDRPGIS}, so we only briefly sketch it.  Assume $\text{PFA}^{+\omega_1}_{H^V_{\omega_2} \notin \text{IA}}$.  Let $\mathbb{Q}^{\text{proper}}_{\theta,\omega_1}$ be the poset from Theorem \ref{thm_TwoPosets}, with $T = \omega_1$.  By that theorem, $H^V_{\omega_2}$ is forced to be indestructibly guessing, and have internal part containing $T$.  Guessing models are never internally approachable, so in particular the poset forces $H^V_{\omega_2} \notin \text{IA}$, and hence the forcing axiom $\text{PFA}^{+\omega_1}_{H^V_{\omega_2} \notin \text{IA}}$ applies to it.  So by Woodin's Lemma \ref{lem_WoodinKeyLemma} there are stationarily many $W \prec (H_{(2^\theta)^+},\in)$ for which there exists a $g$ that is $W$-generic and stationary correct, and Corollary \ref{cor_ControlInternal} ensures that such $W$ are indestructibly guessing and have internal part containing $T=\omega_1$, hence internally club.  The proof that  $W$ is internally, diagonally reflecting is identical to the corresponding proof in Section \ref{subsub_PFAnotGIC}.

\subsubsection{Proof of Theorem \ref{thm_MMplusnotISimpliesORD_DRP}}

Assume $\text{MM}^{+\omega_1}_{H^V_{\omega_2} \notin \text{IS}}$, and consider the poset $\mathbb{Q}^{\omega_1 \text{-SSP}}_{\emptyset, \theta}$ from Theorem \ref{thm_TwoPosets} (taking the $T$ from that theorem to be the empty set); that theorem yields that $H^V_{\omega_2}$ is forced to be indestructibly guessing and have internal part exactly $T = \emptyset$, and hence to \emph{not} be in $\text{IS}$.  So $\mathbb{Q}^{\omega_1 \text{-SSP}}_{\emptyset, \theta}$ is a poset to which $\text{MM}^{+\omega_1}_{H^V_{\omega_2} \notin \text{IS}}$ is applicable.  By Woodin's Lemma \ref{lem_WoodinKeyLemma} there are stationarily many $W \in \wp^*_{\omega_2}(H_{(2^\theta)^+})$ such that there is some $(W,\mathbb{Q}^{\omega_1 \text{-SSP}}_{\emptyset, \theta})$-generic filter that is stationary correct (about names for stationary subsets of $\omega_1$ lying in $W$).  By Corollary \ref{cor_ControlInternal}, it follows that $R \cap \text{sup}(W \cap \theta)$ is stationary in $\text{sup}(W \cap \theta)$ whenever $R \in W$ and $R$ is stationary in $\theta \cap \text{cof}(\omega)$.

\subsection{Preservation results}

Here we prove Theorems \ref{thm_Cox_PreservePFAnegIC} and  \ref{thm_Cox_PreservePFAnegIA}.

\subsubsection{Proof of Theorem \ref{thm_Cox_PreservePFAnegIC}}\label{sec_proofPFAnotICpreserved}

  Assume $V$ is a model of $\text{PFA}^{+\omega_1}_{H^V_{\omega_2} \notin \text{IC}}$.  By Theorem \ref{thm_PFAplusNotGICimpliesDRPGIS}, $\text{DRP}_{\text{internal, GIS}}$ holds, so in particular WRP holds.   Then by Foreman-Magidor-Shelah~\cite{MR924672}, $2^{\omega_1} = \omega_2$.\footnote{One could probably modify the arguments of Velickovic and Todorcevic to prove $2^{\omega_1} = \omega_2$ already follows from the ``non-plus version" $\text{PFA}_{H^V_{\omega_2} \notin \text{IC}}$.  We did not check this, however.} Fix a bijection $\Phi: \omega_2 \to H_{\omega_2}$ and define
\[
\text{IC}^*:= \{ \gamma \in S^2_1 \ : \ \Phi[ \gamma] \in \text{IC} \}.
\]
It is routine to see that for all but nonstationarily many $W \in \text{IC} \cap \wp_{\omega_2}(H_{\omega_2})$, $W$ is of the form $\Phi [W \cap \omega_2]$.

As in Krueger~\cite{MR2674000}, let $\mathbb{P}_{\text{nrIC}}$ be the following partial order (nrIC stands for ``nonreflecting in IC"):  conditions are functions $f: \alpha \cap \text{cof}(\omega) \to 2$ for some $\alpha < \omega_2$ such that for every $\gamma \in \text{IC}^*$, the set 
\[
\{ \xi < \gamma \ : \ f(\xi) =1 \}
\]
is nonstationary in $\gamma$.  The ordering is by function extension.  Then:
\begin{enumerate}
 \item $\mathbb{P}_{\text{nrIC}}$ is $\sigma$-closed and $<\omega_2$ distributive (Lemma 4.1 of \cite{MR2674000}), so in particular
 \[
 H^V_{\omega_2} = H^{V^{\mathbb{P}_{\text{nrIC}}}}_{\omega_2}
 \]
 and
 \[
\Gamma:= \Big( \text{IC} \cap \wp_{\omega_2}(H_{\omega_2}) \Big)^V \text{ is equal to } \Big( \text{IC} \cap \wp_{\omega_2}(H_{\omega_2}) \Big)^{V^{\mathbb{P}_{\text{nrIC}}}}.
 \]

 \item If $G$ is $(V,\mathbb{P}_{\text{nrIC}})$-generic, then by identifying $G$ with $\{ \xi < \omega_1 \ : \ G(\xi) = 1 \}$, in $V[G]$ we have:
\begin{itemize}
 \item $G$ is a stationary subset of $S^2_0$ (Lemma 4.2 of \cite{MR2674000});
 \item $G \cap \gamma$ is nonstationary for all $\gamma \in \text{IC}^*$;
 \item $\text{RP}_{\text{IC}}$ fails in $V[G]$.  We briefly sketch the argument.  Set 
 \[
 \widetilde{G}:= \{ z \in [\omega_2]^\omega \ : \ \text{sup}(z) \in G \}.
 \] 
 Then $\widetilde{G}$ is a stationary subset of $[\omega_2]^\omega$.  If $W \in \Gamma$ (as defined above) and 
 \[
 \text{Sk}^{(H_{\omega_3}[G],\in, \Phi, G, \Delta)} (W) \cap H_{\omega_2} = W
 \]
  where $\Delta$ is some wellorder of $H^{V[G]}_{\omega_3}$, then since $G \cap (W \cap \omega_2)$ is nonstationary and $\text{cf}(W \cap \omega_2) = \omega_1$, it is straightforward to see that $\widetilde{G} \cap [W]^\omega$ is nonstationary.  This shows that $\widetilde{G}$ can reflect to at most \emph{nonstationarily} many members of $\Gamma$.  But if $\text{RP}_{\text{IC}}$ held, then any stationary subset of $[\omega_2]^\omega$ would reflect to stationarily many members of $\Gamma$ (see Foreman-Todorcevic~\cite{MR2115072}, Lemma 8).
\end{itemize}
\end{enumerate}

We use Theorem \ref{thm_PreservationFA} to prove that $\mathbb{P}_{\text{nrIC}}$ preserves $\text{PFA}^{+\omega_1}_{H^V_{\omega_2} \notin \text{IC}}$.  Let $\dot{\mathbb{Q}}$ be a $\mathbb{P}_{\text{nrIC}}$-name for a proper poset that forces ``the ground model's $H_{\omega_2}$ is not in IC"; i.e.\ that $H^{V^{\mathbb{P}_{\text{nrIC}}}}_{\omega_2} \notin \text{IC}$.  Since $H^V_{\omega_2} =H^{V^{\mathbb{P}_{\text{nrIC}}}}_{\omega_2}$,  $\mathbb{P}_{\text{nrIC}}*\dot{\mathbb{Q}}$ forces $H^V_{\omega_2} \notin \text{IC}$.  Let $\dot{\mathbb{R}}$ be the $\mathbb{P}_{\text{nrIC}}*\dot{\mathbb{Q}}$ name for the trivial poset, which is trivially forced to be $\omega_1$-SSP.  Then $\mathbb{P}_{\text{nrIC}}*\dot{\mathbb{Q}} (* \dot{\mathbb{R}})$ is a poset to which the axiom $\text{PFA}^{+\omega_1}_{H^V_{\omega_2} \notin \text{IC}}$ applies.  This verifies clause \ref{item_3stepinGamma} of  Theorem \ref{thm_PreservationFA}.

We now verify clause \ref{item_ClauseWhenever} of Theorem \ref{thm_PreservationFA}.  Suppose $j: V \to N$ is a generic elementary embedding satisfying the assumptions of that clause, where in particular $G*H(*K) \in N$ is generic over $V$ for $\mathbb{P}_{\text{nrIC}}*\dot{\mathbb{Q}} (*\dot{\mathbb{R}})$ and correct with respect to stationary subsets of $\omega_1$.  In particular, $H^V_\theta[G*H]$ is correct (from $N$'s point of view) about the fact that the external part of $H^V_{\omega_2}$ is stationary (recall from above that $\mathbb{P}_{\text{nrIC}}*\dot{\mathbb{Q}}$ forces $H^V_{\omega_2} \notin \text{IC}$). In other words, $H^V_{\omega_2} \notin \text{IC}^N$.  Now $j(\Phi)[\omega_2^V] = \text{range}(\Phi) = H^V_{\omega_2}$, so $\omega_2^V \notin j(\text{IC}^*)$.  Hence $j[G] = G$ is a condition in $j(\mathbb{P}_{\text{nrIC}})$, since by the definition of the forcing $\mathbb{P}_{\text{nrIC}}$, there are no requirements whatsoever at points which are not in $\text{IC}^*$.  This verifies clause \ref{item_ClauseWhenever} of Theorem \ref{thm_PreservationFA}, and hence by that theorem, $\text{PFA}^{+\omega_1}_{H^V_{\omega_2} \notin \text{IC}}$ holds in $V^{\mathbb{P}_{\text{nrIC}}}$.

Finally, suppose in addition that PFA held in the ground model.  The poset $\mathbb{P}_{\text{nrIC}}$ is $\omega_1 + 1$-tactically closed, as in Yoshinobu~\cite{MR3037549}; the argument is virtually identical to the proof of Example 1 on page 751 of \cite{MR3037549}.  Hence, $\mathbb{P}_{\text{nrIC}}$ preserves PFA by the main theorem of \cite{MR3037549}.  Alternatively, one could use Theorem \ref{thm_PreservationFA} to run basically the same argument as the Beaudoin-Magidor example in Section \ref{sec_BedMagExample} above.

\subsubsection{Proof of Theorem \ref{thm_Cox_PreservePFAnegIA}}

This proof is very similar to the proof of Theorem \ref{thm_Cox_PreservePFAnegIC} above, so we only briefly sketch it.  Assume that the ground model $V$ satisfies $\text{PFA}^{+\omega_1}_{H^V_{\omega_2} \notin \text{IA}}$, and define $\text{IA}^*$ and $\mathbb{P}_{\text{nrIA}}$ similarly to the way that $\text{IC}^*$ and $\mathbb{P}_{\text{nrIC}}$ were defined in Section \ref{sec_proofPFAnotICpreserved}, with IA playing the role here that IC played there.  Then $\mathbb{P}_{\text{nrIA}}$ has the same properties that $\mathbb{P}_{\text{nrIC}}$ had, with the obvious modifications. In particular, if $G$ is generic for $\mathbb{P}_{\text{nrIA}}$ then $\text{RP}_{\text{IA}}(\omega_2)$ fails in $V[G]$.

Again, use Theorem \ref{thm_PreservationFA} to show that $\mathbb{P}_{\text{nrIA}}$ preserves $\text{PFA}^{+\omega_1}_{H^V_{\omega_2} \notin \text{IA}}$.  This is basically the same argument as above, except in the current situation, the reason that $H^V_{\omega_2} \notin \text{IA}^N$ is because it is indestructibly guessing, and guessing models are never internally approachable.  This ensures, just as in the proof above, that $j[G] = G$ is a condition in $j(\mathbb{P}_{\text{nrIA}})$.  The rest is identical to the proof above, so we omit the argument.

\subsubsection{Proof of Theorem \ref{thm_MM_notIS_notImplyRPIU}}

Assume that $V$ satisfies $\text{MM}^{+\omega_1}_{H^V_{\omega_2} \notin \text{IS}}$.  Let $\mathbb{P}=\mathbb{P}\big( [\omega_2]^\omega \big)$ be the poset from Section 1 of Aspero-Krueger-Yoshinobu~\cite{MR2567928}.  Conditions are $\omega_1$-sized subsets $p$ of $[\omega_2]^\omega$ such that $p \cap [W]^\omega$ is nonstationary in $[W]^\omega$ whenever $W \in [\omega_2]^{\omega_1}$.  A condition $q$ is stronger than $p$ if $q \supseteq p$ and for every $y \in q \setminus p$, $y$ is not a subset of $\bigcup p$.  Section 1 of \cite{MR2567928} proves that:
\begin{itemize}
 \item $\mathbb{P}$ is $<\omega_2$ strategically closed; in particular it is $\omega_1$-SSP, and $H^V_{\omega_2} = H^{V^{\mathbb{P}}}_{\omega_2}$;
 \item $\mathbb{P}$ adds a stationary subset of $[\omega_2]^\omega$ that does not reflect to any set of size $\omega_1$; hence $\text{RP}(\omega_2)$ fails in the extension.
\end{itemize}

It remains to verify that $\mathbb{P}$ preserves $\text{MM}^{+\omega_1}_{H^V_{\omega_2} \notin \text{IS}}$; again we use Theorem \ref{thm_PreservationFA}.  Let $\dot{\mathbb{Q}}$ be a $\mathbb{P}$-name for an $\omega_1$-SSP poset that forces the ground model is not in IS, and let $\dot{\mathbb{R}}$ be the $\mathbb{P}*\dot{\mathbb{Q}}$ name for the trivial forcing, which is of course $\omega_1$-SSP.  Since $H^V_{\omega_2} = H^{V^{\mathbb{P}}}_{\omega_2}$, $\mathbb{P}*\dot{\mathbb{Q}} (*\dot{\mathbb{R}})$ forces $H^V_{\omega_2} \notin \text{IS}$ and is $\omega_1$-SSP, so clause \ref{item_3stepinGamma} of Theorem \ref{thm_PreservationFA} is satisfied.  Now suppose $j: V \to N$ is as in the hypotheses of clause \ref{item_ClauseWhenever} of Theorem \ref{thm_PreservationFA} (using the notation from that clause).  Then $V \cap [\omega_2^V]^\omega$ is nonstationary in $H^V_\theta[G*H (*K)] = H^V_\theta[G*H]$, and so $G$, being the generic subset of $V \cap [\omega_2^V]^\omega$ added by $\mathbb{P}$, is also nonstationary in $H^V_\theta[G*H]$.  Hence $j[G] = G$ is nonstationary in $[\omega_2^V]^\omega$ from the point of view of $N$.  Moreover, since $G$ is nonreflecting in $V[G]$, then $G \cap [\gamma]^\omega$ is nonstationary for every $\gamma < \omega_2^V$, and this is upward absolute from $H^V_\theta[G]$ to $N$.  So $j[G] = G$ is a condition in $j(\mathbb{P})$.  To see that it is stronger than every condition in $j[G]$, consider some $p \in G$.  Then $j[G]=G \supset p = j(p)$, and if $y \in G \setminus p$ then there is some $q \in G$ with $y \in q$, and hence $y \in q \setminus p$.  Without loss of generality we can assume $q$ is stronger than $p$, and so $y$ is not a subset of $\bigcup p$, by definition of the forcing.  This shows that $j[G] = G$ is stronger than every condition in $j[G]$, and hence the assumptions of clause \ref{item_ClauseWhenever} of Theorem \ref{thm_PreservationFA} is satisfied, and so $V^{\mathbb{P}}$ satisfies $\text{MM}^{+\omega_1}_{H^V_{\omega_2} \notin \text{IS}}$.

\section{Closing Remarks}\label{sec_ClosingRemarks}

As mentioned in the introduction, we do not know if the implication
\[
\text{RP}_{\text{internal}} \ \implies \ \text{RP}_{\text{IS}}
\]
from the Fuchino-Usuba Theorem \ref{thm_FuchinoUsubaImplications}, or the implication 
\[
\text{RP}_{\text{IS}} \implies \text{RP}_{\text{IU}}
\]
 from \eqref{eq_RP_implic}, can be reversed.  These questions are very similar, since separating either implication would require some stationary sets to reflect \emph{only} to the external parts of members of $\wp^*_{\omega_2}(V)$.   We suspect a solution to one of them would likely solve the other, and conjecture that neither can be reversed.  The following theorem lends some support to this conjecture:
\begin{theorem}
Martin's Maximum implies that there is a stationary subset $S$ of $[H_{\omega_2}]^\omega$ such that for stationarily many $W \in \text{IS} \cap \wp^*_{\omega_2}(H_{\omega_2})$:  $S$ reflects to $W$, but not internally.\footnote{I.e.\ $W \in \text{IS}$ and $S \cap [W]^\omega$ is stationary, but $S \cap W \cap [W]^\omega$ is nonstationary.}
\end{theorem}
\begin{proof}
Fix a stationary, costationary subset $T$ of $\omega_1$, and set $S_{T^c}:= [H_{\omega_2}]^\omega \searrow T^c$.  Let $\mathbb{Q}_T:=\mathbb{Q}^{\omega_1 \text{-SSP}}_{T,\omega_2}$ be the $\omega_1$-SSP poset from Theorem \ref{thm_TwoPosets}.  By Woodin's Lemma \ref{lem_WoodinKeyLemma}, there are stationarily many $W \in \wp^*_{\omega_2}(H_{\omega_3})$ for which there exists a $\big( W, \mathbb{Q}_T \big)$-generic filter. Let $R$ be this stationary set.  By Corollary \ref{cor_ControlInternal}, if $W \in R$ then $\bar{W}:=W \cap H_{\omega_2}$ has internal part exactly $T$, which implies that $\bar{W} \in \text{IS}$ and that $S_{T^c} \cap \bar{W} \cap [\bar{W}]^\omega$ is nonstationary.  On the other hand, stationarity of $T^c$ ensures that $S_{T^c} \cap [\bar{W}]^\omega = ([H_{\omega_2}]^\omega \searrow T^c) \cap [\bar{W}]^\omega$ is stationary, since in fact $S_{T^c} \cap [Z]^\omega$ is stationary for any set $Z \subseteq H_{\omega_2}$ such that $\omega_1 \subset Z$.  Then $R_1:= \{ W \cap H_{\omega_2} \ : \ W \in R \}$ is as required.  
\end{proof}


\bib{MR2567928}{article}{
  author={Asper\'o, David},
  author={Krueger, John},
  author={Yoshinobu, Yasuo},
  title={Dense non-reflection for stationary collections of countable sets},
  journal={Ann. Pure Appl. Logic},
  volume={161},
  date={2009},
  number={1},
  pages={94--108},
  issn={0168-0072},
  review={\MR {2567928}},
  doi={10.1016/j.apal.2009.06.002},
}

\bib{MR1099782}{article}{
  author={Beaudoin, Robert E.},
  title={The proper forcing axiom and stationary set reflection},
  journal={Pacific J. Math.},
  volume={149},
  date={1991},
  number={1},
  pages={13--24},
  issn={0030-8730},
  review={\MR {1099782 (92b:03054)}},
}

\bib{Cox_Nonreasonable}{article}{
  author={Cox, Sean D.},
  title={Chang's conjecture and semiproperness of nonreasonable posets},
  journal={Monatsh. Math.},
  volume={187},
  date={2018},
  number={4},
  pages={617--633},
}

\bib{Cox_FA_Ideals}{article}{
  author={Cox, Sean},
  title={PFA and ideals on $\omega _2$ whose associated forcings are proper},
  journal={Notre Dame Journal of Formal Logic},
  volume={53},
  number={3},
  date={2012},
  pages={397--412},
}

\bib{DRP}{article}{
  author={Cox, Sean},
  title={The Diagonal Reflection Principle},
  journal={Proceedings of the American Mathematical Society},
  volume={140},
  number={8},
  date={2012},
  pages={2893--2902},
}

\bib{Cox_Krueger_2}{article}{
  author={Cox, Sean},
  author={Krueger, John},
  title={Indestructible guessing models and the continuum},
  journal={Fund. Math.},
  volume={239},
  date={2017},
  number={3},
  pages={221--258},
}

\bib{MR2576698}{article}{
  author={Doebler, Philipp},
  author={Schindler, Ralf},
  title={$\Pi _2$ consequences of BMM plus $\text {NS}_{\omega _1}$ is precipitous and the semiproperness of stationary set preserving forcings},
  journal={Math. Res. Lett.},
  volume={16},
  date={2009},
  number={5},
  pages={797--815},
  issn={1073-2780},
}

\bib{MR1668171}{article}{
  author={Feng, Qi},
  author={Jech, Thomas},
  title={Projective stationary sets and a strong reflection principle},
  journal={J. London Math. Soc. (2)},
  volume={58},
  date={1998},
  number={2},
  pages={271--283},
  issn={0024-6107},
  review={\MR {1668171 (2000b:03166)}},
}

\bib{MattHandbook}{article}{
  author={Foreman, Matthew},
  title={Ideals and generic elementary embeddings},
  conference={ title={Handbook of set theory. Vols. 1, 2, 3}, },
  book={ publisher={Springer, Dordrecht}, },
  date={2010},
  pages={885--1147},
}

\bib{MR1450520}{article}{
  author={Foreman, Matthew},
  author={Magidor, Menachem},
  title={A very weak square principle},
  journal={J. Symbolic Logic},
  volume={62},
  date={1997},
  number={1},
  pages={175--196},
}

\bib{MR1846032}{article}{
  author={Foreman, Matthew},
  author={Magidor, Menachem},
  title={Mutually stationary sequences of sets and the non-saturation of the non-stationary ideal on $P_\varkappa (\lambda )$},
  journal={Acta Math.},
  volume={186},
  date={2001},
  number={2},
  pages={271--300},
  issn={0001-5962},
  review={\MR {1846032 (2002g:03094)}},
  doi={10.1007/BF02401842},
}

\bib{MR924672}{article}{
  author={Foreman, M.},
  author={Magidor, M.},
  author={Shelah, S.},
  title={Martin's maximum, saturated ideals, and nonregular ultrafilters. I},
  journal={Ann. of Math. (2)},
  volume={127},
  date={1988},
  number={1},
  pages={1--47},
}

\bib{MR2115072}{article}{
  author={Foreman, Matthew},
  author={Todorcevic, Stevo},
  title={A new L\"owenheim-Skolem theorem},
  journal={Trans. Amer. Math. Soc.},
  volume={357},
  date={2005},
  number={5},
  pages={1693--1715 (electronic)},
  issn={0002-9947},
  review={\MR {2115072 (2005m:03064)}},
  doi={10.1090/S0002-9947-04-03445-2},
}

\bib{MR2276627}{article}{
  author={Friedman, Sy-David},
  author={Krueger, John},
  title={Thin stationary sets and disjoint club sequences},
  journal={Trans. Amer. Math. Soc.},
  volume={359},
  date={2007},
  number={5},
  pages={2407--2420},
}

\bib{FuchinoUsuba}{article}{
  author={Fuchino, Saka{\'e}},
  author={Usuba, Toshimichi},
  title={A reflection principle formulated in terms of games},
  journal={RIMS Kokyuroku},
  number={1895},
  pages={37--47},
}

\bib{MR820120}{article}{
  author={Gitik, Moti},
  title={Nonsplitting subset of ${\scr P}\sb \kappa (\kappa \sp +)$},
  journal={J. Symbolic Logic},
  volume={50},
  date={1985},
  number={4},
  pages={881--894 (1986)},
}

\bib{HamkinsSeabold}{article}{
  title={Well-founded Boolean ultrapowers as large cardinal embeddings},
  author={Hamkins, Joel David},
  author={Seabold, Daniel Evan},
  journal={arXiv preprint arXiv:1206.6075},
  year={2012},
}

\bib{MR1940513}{book}{
  author={Jech, Thomas},
  title={Set theory},
  series={Springer Monographs in Mathematics},
  note={The third millennium edition, revised and expanded},
  publisher={Springer-Verlag},
  place={Berlin},
  date={2003},
  pages={xiv+769},
  isbn={3-540-44085-2},
  review={\MR {1940513 (2004g:03071)}},
}

\bib{MR2050172}{article}{
  author={K{\"o}nig, Bernhard},
  author={Yoshinobu, Yasuo},
  title={Fragments of Martin's maximum in generic extensions},
  journal={MLQ Math. Log. Q.},
  volume={50},
  date={2004},
  number={3},
  pages={297--302},
  issn={0942-5616},
  review={\MR {2050172 (2005b:03117)}},
  doi={10.1002/malq.200410101},
}

\bib{MR2674000}{article}{
  author={Krueger, John},
  title={Internal approachability and reflection},
  journal={J. Math. Log.},
  volume={8},
  date={2008},
  number={1},
  pages={23--39},
  issn={0219-0613},
  review={\MR {2674000 (2011i:03046)}},
  doi={10.1142/S0219061308000701},
}

\bib{MR2332607}{article}{
  author={Krueger, John},
  title={Internally club and approachable},
  journal={Adv. Math.},
  volume={213},
  date={2007},
  number={2},
  pages={734--740},
  issn={0001-8708},
  review={\MR {2332607 (2008e:03079)}},
}

\bib{MR2502487}{article}{
  author={Krueger, John},
  title={Some applications of mixed support iterations},
  journal={Ann. Pure Appl. Logic},
  volume={158},
  date={2009},
  number={1-2},
  pages={40--57},
  issn={0168-0072},
  review={\MR {2502487}},
  doi={10.1016/j.apal.2008.09.024},
}

\bib{MR1782117}{article}{
  author={Larson, Paul},
  title={Separating stationary reflection principles},
  journal={J. Symbolic Logic},
  volume={65},
  date={2000},
  number={1},
  pages={247--258},
  issn={0022-4812},
  review={\MR {1782117 (2001k:03094)}},
  doi={10.2307/2586534},
}

\bib{MR2452816}{article}{
  author={Mitchell, William J.},
  title={$I[\omega _2]$ can be the nonstationary ideal on ${\rm Cof}(\omega _1)$},
  journal={Trans. Amer. Math. Soc.},
  volume={361},
  date={2009},
  number={2},
  pages={561--601},
  issn={0002-9947},
  review={\MR {2452816 (2009m:03081)}},
  doi={10.1090/S0002-9947-08-04664-3},
}

\bib{MR3201836}{article}{
  author={Neeman, Itay},
  title={Forcing with sequences of models of two types},
  journal={Notre Dame J. Form. Log.},
  volume={55},
  date={2014},
  number={2},
  pages={265--298},
}

\bib{MR2387938}{article}{
  author={Sakai, Hiroshi},
  title={Semistationary and stationary reflection},
  journal={J. Symbolic Logic},
  volume={73},
  date={2008},
  number={1},
  pages={181--192},
  issn={0022-4812},
  review={\MR {2387938}},
  doi={10.2178/jsl/1208358748},
}

\bib{MR1623206}{book}{
  author={Shelah, Saharon},
  title={Proper and improper forcing},
  series={Perspectives in Mathematical Logic},
  edition={2},
  publisher={Springer-Verlag},
  place={Berlin},
  date={1998},
  pages={xlviii+1020},
  isbn={3-540-51700-6},
  review={\MR {1623206 (98m:03002)}},
}

\bib{MR763902}{article}{
  author={Todor{\v {c}}evi{\'c}, Stevo},
  title={A note on the proper forcing axiom},
  conference={ title={Axiomatic set theory}, address={Boulder, Colo.}, date={1983}, },
  book={ series={Contemp. Math.}, volume={31}, publisher={Amer. Math. Soc., Providence, RI}, },
  date={1984},
  pages={209--218},
  review={\MR {763902 (86f:03089)}},
  doi={10.1090/conm/031/763902},
}

\bib{MR1261218}{article}{
  author={Todor{\v {c}}evi{\'c}, Stevo},
  title={Conjectures of Rado and Chang and cardinal arithmetic},
  conference={ title={Finite and infinite combinatorics in sets and logic (Banff, AB, 1991)}, },
  book={ series={NATO Adv. Sci. Inst. Ser. C Math. Phys. Sci.}, volume={411}, publisher={Kluwer Acad. Publ., Dordrecht}, },
  date={1993},
  pages={385--398},
  review={\MR {1261218}},
}

\bib{MR1174395}{article}{
  author={Veli{\v {c}}kovi{\'c}, Boban},
  title={Forcing axioms and stationary sets},
  journal={Adv. Math.},
  volume={94},
  date={1992},
  number={2},
  pages={256--284},
  issn={0001-8708},
  review={\MR {1174395 (93k:03045)}},
  doi={10.1016/0001-8708(92)90038-M},
}

\bib{Viale_ChangModel}{article}{
  author={Viale, Matteo},
  title={Category forcings, $\text {MM}^{+++}$, and generic absoluteness for the theory of strong forcing axioms},
  journal={J. Amer. Math. Soc.},
  volume={29},
  date={2016},
  number={3},
  pages={675--728},
  issn={0894-0347},
}

\bib{Viale_GuessingModel}{article}{
  author={Viale, Matteo},
  title={Guessing models and generalized {L}aver diamond},
  journal={Ann. Pure Appl. Logic},
  fjournal={Annals of Pure and Applied Logic},
  volume={163},
  year={2012},
  number={11},
  pages={1660--1678},
  issn={0168-0072},
  coden={APALD7},
  mrclass={Preliminary Data},
  mrnumber={2959666},
  doi={10.1016/j.apal.2011.12.015},
  url={http://dx.doi.org/10.1016/j.apal.2011.12.015},
}

\bib{VW_ISP}{article}{
  author={Viale, Matteo},
  author={Wei{\ss }, Christoph},
  title={On the consistency strength of the proper forcing axiom},
  journal={Adv. Math.},
  volume={228},
  date={2011},
  number={5},
  pages={2672--2687},
}

\bib{MR2723878}{book}{
  author={Woodin, W. Hugh},
  title={The axiom of determinacy, forcing axioms, and the nonstationary ideal},
  series={de Gruyter Series in Logic and its Applications},
  volume={1},
  edition={Second revised edition},
  publisher={Walter de Gruyter GmbH \& Co. KG, Berlin},
  date={2010},
  pages={vi+852},
  isbn={978-3-11-019702-0},
  review={\MR {2723878}},
  doi={10.1515/9783110213171},
}

\bib{MR3037549}{article}{
  author={Yoshinobu, Yasuo},
  title={Operations, climbability and the proper forcing axiom},
  journal={Ann. Pure Appl. Logic},
  volume={164},
  date={2013},
  number={7-8},
  pages={749--762},
  issn={0168-0072},
  review={\MR {3037549}},
  doi={10.1016/j.apal.2012.11.010},
}



\begin{bibdiv}
\begin{biblist}
\bibselect{../../../MasterBibliography/Bibliography}
\end{biblist}
\end{bibdiv}
\end{document}